\documentclass[11pt]{article}
\usepackage{graphicx,color,xcolor} 
\usepackage{float}
\usepackage{placeins}

\usepackage{amsmath,amssymb,amsfonts,amsthm,mathrsfs}
\usepackage[title]{appendix}
\usepackage{hyperref}

\definecolor{revisionpurple}{HTML}{7030A0}

\newtheorem{thm}{Theorem}

\newtheorem{prop}[thm]{Proposition}
\newtheorem{defi}[thm]{Definition}

\newtheorem{coro}[thm]{Corollary} 
\newtheorem{remark}[thm]{Remark}

\counterwithin{thm}{section}
\counterwithin{equation}{section}
\counterwithin{figure}{section}

\newcommand{\cA}{\mathcal A}

\newcommand{\cK}{\mathcal K} \newcommand{\cL}{\mathcal L}
 
 \newcommand{\cP}{\mathcal P} 
 \newcommand{\cR}{\mathcal R}
 \newcommand{\cT}{\mathcal T}

\begin{document}
\title{Dirac Points in High-Contrast Dielectric Honeycomb Structures}
\author{Yi Huang\thanks{School of Mathematical Sciences, Zhejiang University, Hangzhou, 310027, China (yi720888@zju.edu.cn, pingliu@zju.edu.cn).}
\and Bowen Li\thanks{Department of Mathematics, City University of Hong Kong, Kowloon Tong, Hong Kong SAR
  (bowen.li@cityu.edu.hk).} \and Ping Liu\footnotemark[1] \thanks{ZJU Center for Interdisciplinary Applied Mathematics, Zhejiang University, Hangzhou, 310027, China}
}
\maketitle

\begin{abstract} 
We study Dirac points for a two-dimensional scalar Helmholtz equation
in a periodic dielectric medium with two penetrable inclusions per
honeycomb cell.  In a joint small-inclusion and high-contrast regime,
we reduce the Bloch eigenvalue problem to a two-dimensional interaction
matrix and derive a two-term asymptotic expansion for the two relevant
eigenvalues, uniformly for quasi-momenta separated from the reciprocal
lattice.  At a vertex \(K\) of the Brillouin zone, rotation and
inversion-conjugation symmetries force these eigenvalues to coincide.
A Schur-complement analysis of the exact operator pencil, together
with a nonvanishing lattice Green-function coefficient, proves that
the resulting double eigenvalue has a nonzero conical splitting and
gives an asymptotic formula for the conical slope in the scaled eigenvalue.
For fixed
sufficiently large contrast, symmetry-preserving changes of variables
and analytic perturbation theory then continue this Dirac point as the
inclusion size varies: it remains nondegenerate except at a relatively
closed discrete set of admissible sizes. 
\end{abstract}
\section{Introduction}

Honeycomb structures provide a natural setting in which lattice symmetry
produces conical band crossings and unusual wave-transport phenomena.  The
connection between honeycomb geometry and electronic dispersion goes back to
Wallace's tight-binding description of graphite \cite{introWallace1947}.  The
experimental observation of massless Dirac behavior in graphene
\cite{introNovoselov2005} and the subsequent development of its electronic
theory \cite{introCastroNeto2009} motivated the search for analogous effects
in classical-wave systems.  In photonics, conical diffraction has been
observed in honeycomb lattices \cite{introPeleg2007}, and the roles of Dirac
cones, symmetry breaking, and edge states have been investigated in photonic
crystals \cite{introOchiaiOnoda2009}.  Proposals for nonreciprocal photonic
edge transport \cite{introHaldaneRaghu2008} and the experimental realization
of photonic Floquet topological insulators \cite{introRechtsman2013} further
demonstrate the importance of understanding and controlling such band
crossings.

The mathematical analysis of Dirac points begins with the Floquet--Bloch
spectrum of periodic differential operators; see \cite{introKuchment2016}
for a general account.  For continuum Schr\"odinger operators, Fefferman and
Weinstein established the existence and perturbative stability of Dirac
points for honeycomb lattice potentials under an appropriate nondegeneracy
condition \cite{introFeffermanWeinstein2012}, and subsequently justified the
effective Dirac dynamics of spectrally localized wave packets
\cite{introFeffermanWeinstein2014}.  Berkolaiko and Comech clarified the role
of unitary and antiunitary symmetries in spectral degeneracies and conical
crossings \cite{introBerkolaikoComech2018}.  Related mechanisms have also
been studied for quantum graphs \cite{introKuchmentPost2007}, point
scatterers \cite{introLee2016}, and continuum honeycomb potentials in the
strong-binding regime \cite{introFeffermanLeeThorpWeinstein2018}.

For classical waves, the geometry and contrast of the material coefficients
introduce additional analytical difficulties.  Dirac points and edge states
for honeycomb elliptic operators were studied by Lee-Thorp, Weinstein, and
Zhu \cite{introLeeThorpWeinsteinZhu2019}, building on methods developed for
continuum honeycomb edge states \cite{introFeffermanLeeThorpWeinstein2016}.
Cassier and Weinstein derived asymptotic expansions for Dirac eigenpairs and
velocities in high-contrast honeycomb media under spectral isolation and
nonvanishing assumptions \cite{introCassierWeinstein2021}.  At subwavelength
scales, Dirac cones and effective Dirac dynamics have been established for
honeycomb arrays of Minnaert bubbles
\cite{introAmmariBubbles2020,introAmmariHomogenization2020}.  Dirac points
have also been proved for honeycomb arrays of impenetrable obstacles
\cite{introLiLinZhang2023}.  Although these problems share the same lattice
symmetries, their spectral scales, limiting operators, and resonance
mechanisms are different.

In this paper, we consider a two-dimensional scalar dielectric medium with
two penetrable inclusions in each cell of a triangular lattice.  The
inclusions are obtained by placing scaled copies of a threefold
rotation-invariant reference domain at the sites \(z_1\) and \(z_2\), defined
below.  They are paired so that the periodic structure has inversion
symmetry, even when the
individual reference domain does not.  The wave field satisfies a weighted
Helmholtz eigenvalue problem with Bloch quasi-momentum \(\xi\).  The parameter
\(\eta\) controls the inclusion size, while \(\tau>0\) is the dielectric
contrast.  Thus the problem contains two interacting singular parameters:
the inclusions become small as \(\eta\to0\), whereas their refractive index
becomes large as \(\tau\to\infty\).  Integral-operator methods have
previously been used to characterize subwavelength resonances of isolated
high-index dielectric particles in scalar wave models
\cite{introAmmariDielectric2019} and in the full Maxwell system
\cite{introAmmariLiZou2023}.  Here we adapt this approach to a
two-dimensional periodic setting, where the logarithmic singularity of the
quasiperiodic Green function determines the leading spectral scale.

We now summarize the main results.  First, in the joint small-inclusion and
high-contrast regime, we prove that the Bloch problem has two characteristic
values near the leading resonant scale.  Writing
\(\widehat\lambda=\tau\lambda\) and \(L=|\log\eta|\), we label the two values by
\(\widehat\lambda_-(\xi)\leq\widehat\lambda_+(\xi)\),
suppressing their dependence on \(\tau,\eta\). Their two-term
asymptotic expansion is
\[
 \widehat\lambda_\pm(\xi)
 =
 \frac{2\pi}{|D|\eta^2L}
 \left[
 1-\frac{2\pi\mu_\mp(\xi)}{|D|L}
 +O\left(\frac1{L^2}+\frac1{\tau\eta^2L^2}\right)
 \right],
\]
uniformly for quasi-momenta separated from the reciprocal lattice, provided
\(1/(\tau\eta^2L)\to0\).  The coefficients \(\mu_\pm(\xi)\) are the
eigenvalues of a two-dimensional interaction matrix containing both the
logarithmic self-interaction of each inclusion and the Green-function
coupling between the two sites.  The remainder retains explicit dependence
on both inclusion size and contrast.

Second, at a Brillouin-zone vertex \(K\), rotation and
inversion-conjugation symmetries force the two characteristic values to
coincide exactly.  We then analyze the exact operator pencil, rather than
differentiating the preceding eigenvalue approximation, and prove that the
double eigenvalue has a nonzero linear splitting in the quasi-momentum.  This
establishes an exact Dirac cone and yields an asymptotic formula for its
slope.  The required nonvanishing condition is reduced to a derivative of
the lattice Green function and is verified by an exponentially convergent
Ewald representation.

Third, we fix a sufficiently large contrast and vary the inclusion size.
Symmetry-preserving changes of variables transfer the moving-interface
problems to fixed reference spaces, where analytic perturbation theory can
be applied in each rotation sector.  Starting from the small-inclusion Dirac
point, we obtain a real-analytic eigenvalue branch on the entire
non-overlapping size interval \((0,\eta_{\mathrm{c}})\).  There is a relatively
closed discrete set \(\mathcal C_\tau\) such that the continued crossing is
a nondegenerate Dirac point for every
\(\eta\in(0,\eta_{\mathrm{c}})\setminus\mathcal C_\tau\).  The ordered band index
of this branch may change at spectral crossings.  Finally, fixed-domain
Nystr\"om computations, cross-checked against a Fourier material-matrix
discretization where feasible, compare the numerical eigenvalues for
three-petal inclusions with the asymptotic prediction.

The main contributions of this work are threefold.  First, the analysis
retains the inclusion size and dielectric contrast as independent
parameters and identifies their distinct contributions to the eigenvalue
remainder.  Second, it separates the approximation of the eigenvalues
from the proof of the exact conical splitting: the latter is obtained from a
Schur-complement reduction of the full analytic operator pencil and a
quantitative nonvanishing argument.  Third, the local small-inclusion
construction is combined with a fixed-domain analytic formulation to
continue the Dirac point beyond the asymptotic regime, with all possible
failures confined to a discrete set of inclusion sizes.

The paper is organized as follows.  Section~2 introduces the honeycomb
geometry, its symmetries, and the Floquet--Bloch formulation.  Section~3
derives the small-inclusion asymptotics and proves the symmetry-protected
degeneracy and conical splitting at \(K\).
Section~4 establishes the analytic continuation of the Dirac point with
respect to the inclusion size.  Section~5 presents numerical tests of the
eigenvalue asymptotics and the conical slope.
Appendices~\ref{app:c-nonzero} and~\ref{app:exact-Dirac-cone}
establish the nonvanishing Green-function coefficient and the exact
conical slope, respectively.

\section{Model and preliminaries}

\subsection{Geometric setting and dielectric model}
As illustrated in Figure~\ref{fig:cellhomo}, we consider an infinite honeycomb structure generated by the triangular lattice
\[
\Lambda:=\mathbb Z \mathbf e_1\oplus \mathbb Z \mathbf e_2
=
\left\{\ell_1\mathbf e_1+\ell_2\mathbf e_2:\ell_1,\ell_2\in\mathbb Z\right\},
\]
where
\[
\mathbf e_1=a\left(\frac{\sqrt3}{2},-\frac12\right)^T,
\qquad
\mathbf e_2=a\left(\frac{\sqrt3}{2},\frac12\right)^T.
\]
Throughout this paper, we set the lattice constant \(a=1\). The fundamental cell of \(\Lambda\) is chosen as
\begin{equation}\label{eq:cellz}
Y:=
\left\{\ell_1\mathbf e_1+\ell_2\mathbf e_2:
\ell_1,\ell_2\in[-1/2,1/2)\right\}.
\end{equation}

We shall use two symmetry transformations on \(\mathbb R^2\): the clockwise rotation \(R\) by angle \(2\pi/3\), and the inversion \(T\) about the origin, given by
\begin{equation}\label{eq:RT-coordinate}
R x
:=
\begin{pmatrix}
-\frac12 & \frac{\sqrt3}{2}\\
-\frac{\sqrt3}{2} & -\frac12
\end{pmatrix}x,
\qquad
Tx:=-x.
\end{equation}
They induce actions on functions by
\begin{equation}\label{eq:RT-function}
(Rf)(x):=f(R^{-1}x),
\qquad
(Tf)(x):=f(T^{-1}x).
\end{equation}
A function \(f\) is called \(R\)-invariant if \(Rf=f\), and \(T\)-invariant if \(Tf=f\).

\begin{figure}[ht]
\begin{center}
\includegraphics[height=4cm]{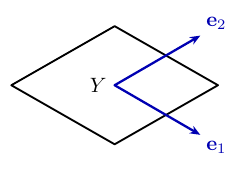}
\hspace{1cm}
\includegraphics[height=5cm]{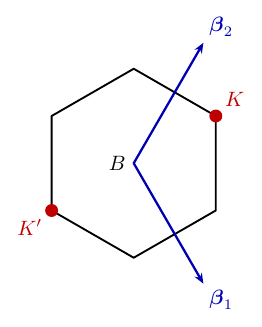}
\caption{The fundamental cell \(Y\) of the triangular lattice \(\Lambda\) (left) and the Brillouin zone \(B\) (right).}
\label{fig:cellhomo}
\end{center}
\end{figure}

The reciprocal lattice \(\Lambda^*\) is defined by
\[
\Lambda^*
:=
\left\{
m_1\boldsymbol\beta_1+m_2\boldsymbol\beta_2:
m_1,m_2\in\mathbb Z
\right\},
\]
where
\begin{equation}\label{eq:reciprocal_vectors}
\boldsymbol\beta_1=
2\pi\left(\frac1{\sqrt3},-1\right)^T,
\qquad
\boldsymbol\beta_2=
2\pi\left(\frac1{\sqrt3},1\right)^T.
\end{equation}
These vectors satisfy
\[
\mathbf e_i\cdot \boldsymbol\beta_j=2\pi\delta_{ij},
\qquad i,j=1,2.
\]
The first Brillouin zone \(B\) is the Wigner--Seitz cell of the reciprocal lattice \(\Lambda^*\). It is a regular hexagon centered at the origin. Its vertices are the high-symmetry points
\[
K,\quad RK,\quad R^2K,
\qquad
K^\prime,\quad RK^\prime,\quad R^2K^\prime,
\]
where
\begin{equation}\label{eq:Kpoints}
K:=2\pi\left(\frac1{\sqrt3},\frac13\right)^T
=
\frac13\boldsymbol\beta_1+\frac23\boldsymbol\beta_2,
\qquad
K^\prime:=-K.
\end{equation}
The points \(K\) and \(K^\prime\) are inequivalent vertices of the Brillouin zone modulo the reciprocal lattice.

\paragraph{Two small inclusions in the fundamental cell}

We now specify the geometry of the inclusions. Let
\begin{equation}\label{eq:z1z2q}
z_1=\frac{\mathbf e_1+\mathbf e_2}{6}
=
\left(\frac{\sqrt3}{6},0\right)^T,
\qquad
z_2=-z_1
=
\left(-\frac{\sqrt3}{6},0\right)^T,
\end{equation}
and let
\begin{equation}\label{eq:q_rotation_center}
q=-\frac{\mathbf e_1+\mathbf e_2}{2}
=
\left(-\frac{\sqrt3}{2},0\right)^T .
\end{equation}
The point \(q\) is the left vertex of the rhombic cell \(Y\).

Let \(D\subset\mathbb R^2\) be a bounded reference domain containing the origin. We assume that \(D\) is invariant under the rotation by \(2\pi/3\) about the origin. For a small parameter \(\eta>0\), we define
\begin{equation}\label{eq:D_eta}
D_{\eta,1}:=z_1+\eta D,
\qquad
D_{\eta,2}:=z_2-\eta D,
\qquad
D_\eta:=D_{\eta,1}\cup D_{\eta,2}.
\end{equation}
The periodic array of inclusions is
\begin{equation}\label{eq:D_eta_per}
D_\eta^{\mathrm{per}}
:=
\bigcup_{\ell\in\Lambda}
\left(D_\eta+\ell\right).
\end{equation}

\begin{figure}[ht]
\begin{center}
\includegraphics[height=8cm]{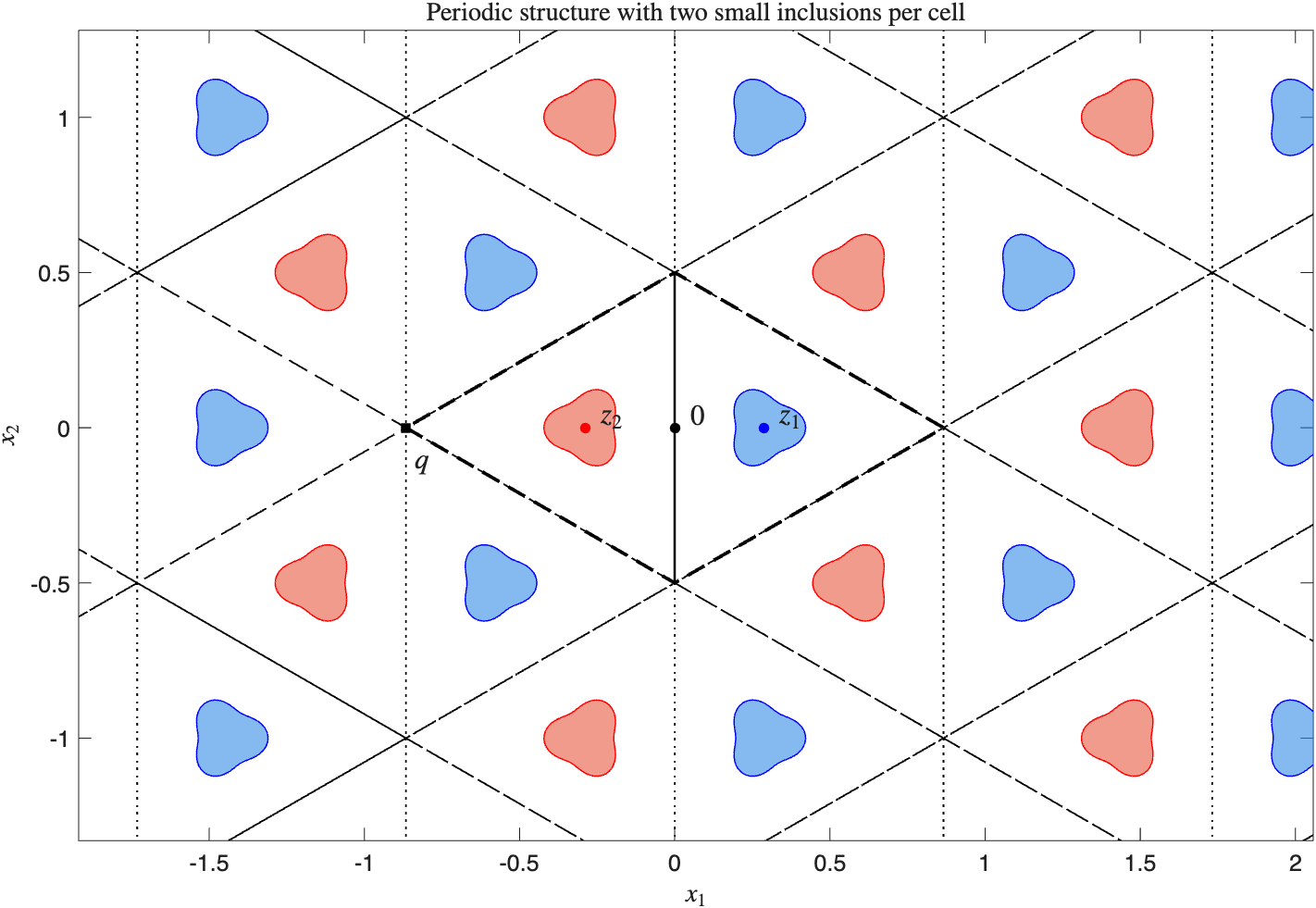}
\caption{Periodic structure with two small inclusions per cell.}
\label{fig:two_inclusions}
\end{center}
\end{figure}

Accordingly, the relative permittivity (the squared refractive index) is taken to be
\begin{equation}\label{eq:n_eta}
n_\eta(x)
=
1+\tau \chi_{D_\eta^{\mathrm{per}}}(x),
\end{equation}
where \(\tau>0\) is the contrast parameter.

We consider the variable-coefficient Helmholtz equation
\begin{equation}\label{equ:variable-coefficient-Helmholtz-equation}
\Delta u(x)+k^2 n_\eta(x)u(x)=0
\qquad \text{in } \mathbb R^2.
\end{equation}

\subsection{Preliminaries}

To formulate the spectral problem, let
\[
L_{n_\eta}^2(\mathbb R^2):=L^2(\mathbb R^2,n_\eta(x)\,dx).
\]
We define the sesquilinear form
\[
\Phi(v,u):=
\int_{\mathbb R^2}\overline{\nabla v}\cdot \nabla u\,dx,
\qquad
u,v\in H^1(\mathbb R^2).
\]
Let \(P\) denote the self-adjoint operator associated with this form on
\(L_{n_\eta}^2(\mathbb R^2)\). Since \(n_\eta\) is \(\Lambda\)-periodic, the
Floquet--Bloch transform decomposes \(P\) into operators acting on a single
cell, one for each quasi-momentum \(\xi\in B\).

For each \(\xi\in B\), the corresponding space of \(\xi\)-quasiperiodic
functions is
\[
L_{n_\eta,\xi}^2(Y):=
\left\{
u|_Y:
u\in L_{\mathrm{loc}}^2(\mathbb R^2),\
u(x+\ell)=e^{i\xi\cdot\ell}u(x)
\ \text{for all }\ell\in\Lambda
\right\}.
\]
On this space, the fiber operator \(P_\xi\) is associated with the form
\[
\int_Y \overline{\nabla v}\cdot \nabla u\,dx
\]
with form domain \(H_\xi^1(Y)\), the space of \(\xi\)-quasiperiodic
\(H^1\) functions on \(Y\). Alternatively, one can work
with periodic functions by removing the Bloch phase. The map
\[
U_\xi u(x):=e^{-i\xi\cdot x}u(x)
\]
is unitary from \(L_{n_\eta,\xi}^2(Y)\) onto \(L_{n_\eta,0}^2(Y)\). Under this
map, \(P_\xi\) becomes the operator \(\widetilde P_\xi\), associated with the
form
\[
\int_Y
\overline{(\nabla+i\xi)v}\cdot(\nabla+i\xi)u\,dx
\]
with periodic \(H^1\) form domain. Thus
\[
\widetilde P_\xi=U_\xi P_\xi U_\xi^{-1}.
\]
In particular, the two fiber representations have the same spectrum:
\[
\sigma(P_\xi)=\sigma(\widetilde P_\xi).
\]
The Floquet--Bloch decomposition therefore gives
\[
\sigma(P)
=
\bigcup_{\xi\in B}\sigma(P_\xi)
=
\bigcup_{\xi\in B}\sigma(\widetilde P_\xi).
\]
Since \(P_\xi\) is a nonnegative self-adjoint elliptic operator with compact resolvent, its spectrum consists of a sequence of eigenvalues
\[
0\leq \lambda_1(\xi)\leq \lambda_2(\xi)\leq \cdots,
\]
counted with multiplicity. The functions \(\lambda_j(\xi)\) are called the band dispersion functions. They are Lipschitz continuous with respect to \(\xi\in B\).

To study these eigenvalues, we compare \(P_\xi\) with the free fiber
Laplacian \(-\Delta_\xi\), defined by the form
\[
\int_Y\overline{\nabla v}\cdot\nabla u\,dx
\]
on \(L^2(Y)\) with \(\xi\)-quasiperiodic \(H^1\) form domain.
For each $k\in \Lambda^*$, define
\[
e_k(x):=|Y|^{-1/2}e^{i(k+\xi)\cdot x}, \qquad x\in Y.
\]
Then $\{e_k\}_{k\in\Lambda^*}$ is a complete orthonormal basis of $L^2(Y)$. Moreover,
\[
e_k\in D(-\Delta_\xi), \qquad
-\Delta_\xi e_k = |k+\xi|^2 e_k,
\qquad \forall k\in \Lambda^*.
\]
Consequently, the spectrum of $-\Delta_\xi$ is purely discrete and is given by \cite[p.~65, Theorem~2.34]{lewin2024spectral}
\[
\sigma(-\Delta_\xi)=\{|k+\xi|^2:\ k\in\Lambda^*\}.
\]
This Fourier basis gives the following expression for the free resolvent
kernel.
\begin{defi}
Let $\lambda\in \mathbb C\setminus \sigma(-\Delta_\xi)$. We define the Green function associated with the fiber operator $-\Delta_\xi$ by
\begin{equation}\label{equ:Gfourier}
G_\lambda^\xi(x,y)
:=
\sum_{k\in\Lambda^*}
\frac{e_k(x)\overline{e_k(y)}}{|k+\xi|^2-\lambda}
=
\frac{1}{|Y|}\sum_{k\in\Lambda^*}
\frac{e^{i(k+\xi)\cdot (x-y)}}{|k+\xi|^2-\lambda}
,
\qquad x,y\in Y.
\end{equation}
The series is understood in \(L^2(Y\times Y)\), as justified in
Proposition~\ref{prop:resolvent_kernel}.
\end{defi}

We will use this kernel to express the perturbation caused by the
inclusions as an integral operator. The Hilbert--Schmidt criterion
provides a convenient way to control that operator.
\begin{defi}
An operator \(T\in\cL(H)\) is Hilbert--Schmidt if
\(\|T\|_2^2:=\operatorname{tr}(T^*T)<\infty\). The class of such
operators is denoted by \(\mathfrak S_2\).
\end{defi}
Every Hilbert--Schmidt operator is compact. We use the following
integral-kernel characterization \cite[p.~210]{reed2012methods}.
\begin{thm}
An operator \(T\in\cL(L^2(Y))\) is Hilbert--Schmidt if and only if
there is a kernel \(K\in L^2(Y\times Y)\) such that
\[
Tf(x)=\int_Y K(x,y)f(y)\,dy
\qquad\text{for every }f\in L^2(Y).
\]
In that case,
\[
\|T\|_2^2=\int_{Y\times Y}|K(x,y)|^2\,dx\,dy.
\]
\end{thm}

\begin{prop}\label{prop:resolvent_kernel}
Let $\lambda\in \mathbb C\setminus \sigma(-\Delta_\xi)$. Then
\[
G_\lambda^\xi \in L^2(Y\times Y).
\]
Moreover, for every $f\in L^2(Y)$, one has
\begin{equation}\label{equ:resolvent_kernel}
(-\Delta_\xi-\lambda)^{-1}f(x)
=
\int_Y G_\lambda^\xi(x,y)f(y)\,dy
\qquad \text{in }L^2(Y).
\end{equation}
In other words, $G_\lambda^\xi$ is the resolvent kernel of $(-\Delta_\xi-\lambda)^{-1}$.
\end{prop}
\begin{proof}
The functions \(e_k(x)\overline{e_k(y)}\), \(k\in\Lambda^*\), are
orthonormal in \(L^2(Y\times Y)\). Since
\[
\sum_{k\in\Lambda^*}\frac{1}{\bigl||k+\xi|^2-\lambda\bigr|^2}<\infty
\]
in two dimensions, the series defining \(G_\lambda^\xi\) converges in
\(L^2(Y\times Y)\). For each finite partial sum, integration against
\(f\) gives the corresponding partial sum of the Fourier expansion of
\((-\Delta_\xi-\lambda)^{-1}f\). By the Cauchy--Schwarz inequality, the
\(L^2(Y\times Y)\) convergence of the kernels implies convergence of
these integrals in \(L^2(Y)\), which proves the claimed kernel
representation.
\end{proof}

For the integral formulation below, take \(\eta\) sufficiently small
that \(\overline{D_\eta}\subset\operatorname{int}Y\).
Set \(\lambda=k^2\) in the Helmholtz equation. On the \(\xi\)-fiber, the
identity \(n_\eta|_Y=1+\tau\chi_{D_\eta}\) then gives
\[
(-\Delta_\xi-\lambda)u
=\tau\lambda\chi_{D_\eta}u.
\]
For \(\lambda\notin\sigma(-\Delta_\xi)\), the free resolvent therefore
reduces the eigenvalue problem to an integral equation on \(D_\eta\).
The following theorem states this equivalence precisely.

\begin{thm}\label{thm:LSequation}
Fix \(\xi\in B\) and \(\lambda\notin\sigma(-\Delta_\xi)\). Then
\(\lambda\in\sigma(P_\xi)\) if and only if the fiber
Lippmann--Schwinger equation
\begin{equation}\label{equ:LS_equation}
    \big(I-\tau\lambda\cT_{D_\eta}^{\xi,\lambda}\big)[u](x)=0,\quad\text{for a.e. }x\in D_\eta,
\end{equation}
admits a nontrivial solution in \(L^2(D_\eta)\), where
\[
\cT_{D_\eta}^{\xi,\lambda}[u](x)
:=\int_{D_\eta}G_\lambda^\xi(x,y)u(y)\,dy.
\]
By Proposition~\ref{prop:resolvent_kernel} and the kernel criterion,
\(\cT_{D_\eta}^{\xi,\lambda}\) is Hilbert--Schmidt on
\(L^2(D_\eta)\). Equivalently, \(\lambda\) is a characteristic value
of the operator function
\[
\cA(\mu,\xi;\tau,\eta):=I-\tau\mu\cT_{D_\eta}^{\xi,\mu}.
\]
\end{thm}
\begin{proof}
The form definitions and \(1\leq n_\eta\leq1+\tau\) imply
\[
D(P_\xi)=D(-\Delta_\xi),\qquad -\Delta_\xi=n_\eta P_\xi.
\]
Suppose first that \(\lambda\in\sigma(P_\xi)\). Since \(P_\xi\) has
compact resolvent, there is a nonzero \(u\in D(P_\xi)\) such that
\[
P_\xi u=\lambda u.
\]
By the definition of \(P_\xi\),
\[
\int_Y \overline{\nabla v}\cdot\nabla u\,dx
=
\lambda\int_Y(1+\tau\chi_{D_\eta})\overline v u\,dx
\qquad \forall v\in H_\xi^1(Y).
\]
Thus, in the distributional sense,
\begin{align} \label{equ:LS1}
(-\Delta_\xi-\lambda)u=\lambda\tau\chi_{D_\eta}u.
\end{align}

Since \(\lambda\notin\sigma(-\Delta_\xi)\), applying the free resolvent
and using \eqref{equ:resolvent_kernel} gives
\begin{equation}\label{equ:LS2}
u(x)
=
\lambda\tau\int_{D_\eta} G_\lambda^\xi(x,y)u(y)\,dy,
\qquad\text{for a.e. }x\in Y.
\end{equation}
The restriction of \(u\) to \(D_\eta\) cannot vanish, since otherwise
\eqref{equ:LS2} would give \(u=0\). Restricting \eqref{equ:LS2} to
\(D_\eta\) proves one implication.

Conversely, suppose \(0\neq u\in L^2(D_\eta)\) solves
\eqref{equ:LS_equation}. Define
\begin{equation}\label{equ:LS3}
U(x):=\lambda\tau\int_{D_\eta} G_\lambda^\xi(x,y)u(y)\,dy,
\qquad x\in Y.
\end{equation}
Since \(U=u\) almost everywhere on \(D_\eta\), it is nonzero. The
resolvent representation gives
\[
U=\lambda\tau(-\Delta_\xi-\lambda)^{-1}(\chi_{D_\eta}U).
\]
Hence \(U\in D(-\Delta_\xi)\) and
\((-\Delta_\xi-\lambda)U=\lambda\tau\chi_{D_\eta}U\).
This is the fiber eigenvalue equation \(P_\xi U=\lambda U\), so
\(\lambda\in\sigma(P_\xi)\).
Restriction to \(D_\eta\) and the extension \eqref{equ:LS3}
are inverse linear maps between the two eigenspaces.
\end{proof}

\section{Small-inclusion asymptotics}

\subsection{Two characteristic values near the leading scale}
Let \(\widehat\lambda:=\tau\lambda\) and \(\epsilon:=1/\tau\).
Define the unitary rescaling
\(U_\eta:L^2(D_\eta)\to L^2(D)\oplus L^2(D)\) by
\(U_\eta f=(f_1,f_2)\), where
\[
f_1(x)=\eta f(z_1+\eta x),
\qquad
f_2(x)=\eta f(z_2-\eta x).
\]
The factor \(\eta\) accounts for the two-dimensional change of
variables. Then we obtain
\begin{equation}\label{equ:LS_operator_matrix}
\begin{aligned}
U_\eta\cA(\epsilon \widehat\lambda,\xi;\epsilon^{-1},\eta)U_\eta^{-1}&=I-\widehat\lambda\eta^2\begin{pmatrix}
    \cT_D^{\xi,\epsilon \widehat\lambda,\eta}&\cR_{+}\\
    \cR_{-} & \cT_D^{-\xi,\epsilon \widehat\lambda,\eta}
\end{pmatrix}\\
&=I-\widehat\lambda\eta^2\begin{pmatrix}
    \cT_D^{\xi,\epsilon \widehat\lambda,\eta}&\\
     & \cT_D^{-\xi,\epsilon \widehat\lambda,\eta}
\end{pmatrix}+O(|\widehat\lambda|\eta^2),
\end{aligned}
\end{equation}
where
\begin{align*}
    \cT_D^{\xi,\epsilon \widehat\lambda,\eta}[f](x)&:=\int_D G^{\xi,\epsilon \widehat\lambda,\eta}(x-y)f(y)\,dy,\\
    G^{\xi,\lambda,\eta}(z)&:=\frac{1}{|Y|}\sum_{k\in\Lambda^*}
\frac{e^{i(k+\xi)\cdot \eta z}}{|k+\xi|^2-\lambda},\\
\cR_{+}[f](x)&:=\int_D
G^{\xi}_{\epsilon\widehat\lambda}
\bigl(z_1+\eta x,z_2-\eta y\bigr)f(y)\,dy,\\
\cR_{-}[f](x)&:=\int_D
G^{\xi}_{\epsilon\widehat\lambda}
\bigl(z_2-\eta x,z_1+\eta y\bigr)f(y)\,dy.
\end{align*}
In the kernel estimates below, we restrict the spectral parameter to
\( |\epsilon\widehat\lambda|\le d_\Xi^2/2\), where
\(d_\Xi:=\inf_{\xi\in\Xi,\,k\in\Lambda^*}|k+\xi|>0\) for a compact
set \(\Xi\). This condition holds near the leading resonant scale when
\(\epsilon/(\eta^2|\log\eta|)\) is sufficiently small.
For \(\xi\) in a compact set separated from \(\Lambda^*\), the
zero-frequency Green function has a logarithmic singularity at the
origin and a smooth regular part. Let \(R_\xi\) denote the value of
that regular part at the origin. For real \(\xi\), the Fourier
representation gives
\(G_0^{-\xi}(z,0)=G_0^\xi(-z,0)
=\overline{G_0^\xi(z,0)}\). Since the logarithmic singularity is
real and even, \(R_\xi\) is real and \(R_{-\xi}=R_\xi\).
Since the regular part is smooth and the change from zero to
\(\epsilon\widehat\lambda\) in the spectral parameter contributes
\(O(\epsilon|\widehat\lambda|)\), we have, for
\(z\in(D-D)\setminus\{0\}\),
\[
G^{\xi,\epsilon\widehat\lambda,\eta}(z)
=
\frac{|\log\eta|}{2\pi}
-\frac{\log|z|}{2\pi}
+R_\xi
+O(\eta+\epsilon|\widehat\lambda|).
\]
The remainder is uniform for \(z\in D-D\) after the logarithmic
singularity is removed, and for \(\xi\) in such compact sets. Hence,
in operator norm,
\[
\cT_D^{\xi,\epsilon \widehat\lambda,\eta}
=\frac{|\log\eta|}{2\pi}\cP+\cK_{\xi}
+O(\eta+\epsilon|\widehat\lambda|),
\]
where
\[
\cP f:=\mathbf 1_D\int_D f(y)\,dy,\qquad
\cK_{\xi}f(x):=
\int_D\left(-\frac{\log|x-y|}{2\pi}+R_{\xi}\right)f(y)\,dy.
\]
In particular, \(\cK_{-\xi}=\cK_\xi\).
Because the two inclusion centers are distinct modulo \(\Lambda\), the off-diagonal
Green kernels are smooth near \((z_1,z_2)\) and \((z_2,z_1)\).
Consequently, also in operator norm,
\[
\cR_+
=G_0^\xi(z_1,z_2)\cP+O(\eta+\epsilon|\widehat\lambda|),
\qquad
\cR_-
=G_0^\xi(z_2,z_1)\cP+O(\eta+\epsilon|\widehat\lambda|).
\]
In particular, \eqref{equ:LS_operator_matrix} gives
\begin{equation}\label{equ:LS_operator_matrix2}
    U_\eta\cA(\epsilon\widehat\lambda,\xi;\epsilon^{-1},\eta)U_\eta^{-1}=I+\frac{1}{2\pi}\widehat\lambda\eta^2\log\eta\begin{pmatrix}
    \cP&\\
     & \cP
\end{pmatrix}+O(|\widehat\lambda|\eta^2)+O(\epsilon|\widehat\lambda|^2\eta^2).
\end{equation}
We regard quasi-momenta as points of \(\mathbb R^2\), with
\(\Lambda^*\)-periodicity, so that neighborhoods of a Brillouin-zone
vertex need not lie inside the chosen hexagonal cell \(B\).
\begin{thm}\label{thm:two_characteristic_values}
Let $\Xi\Subset\mathbb R^2$ satisfy
\[
\inf_{\xi\in\Xi}\inf_{k\in\Lambda^*}|k+\xi|>0,
\]
and set
\[
\epsilon=\frac1\tau,
\qquad
\widehat{\mathcal A}(\widehat\lambda,\xi;\epsilon,\eta)
:=
\mathcal A(\epsilon\widehat\lambda,\xi;\epsilon^{-1},\eta).
\]
If
\[
\eta\to0,
\qquad
\frac{\epsilon}{\eta^2|\log\eta|}\to0,
\]
then, for every $\xi\in\Xi$, there exist two characteristic values
$\widehat\lambda_1(\xi)$ and $\widehat\lambda_2(\xi)$, counted with
algebraic multiplicity, near
\[
\frac{2\pi}{|D|\eta^2|\log\eta|},
\]
such that
\[
\widehat\lambda_{1,2}(\xi)
=
\frac{2\pi}{|D|\eta^2|\log\eta|}
\left[
1
+
O\left(\frac1{|\log\eta|}\right)
+
O\left(
\frac{\epsilon}{\eta^2|\log\eta|^2}
\right)
\right].
\]
The estimates are uniform with respect to $\xi\in\Xi$.

\end{thm}

\begin{proof}
Define
\[
\Pi
:=
\frac1{|D|}
\begin{pmatrix}
\mathcal P&0\\
0&\mathcal P
\end{pmatrix},\quad 
\nu
:=
\frac{|D|}{2\pi}
\widehat\lambda\eta^2|\log\eta|.
\]
Therefore, by \eqref{equ:LS_operator_matrix2},
\[
U_\eta
\widehat{\mathcal A}
(\widehat\lambda,\xi;\epsilon,\eta)
U_\eta^{-1}
=
I-\nu\Pi
+
\mathcal E_{\epsilon,\eta}^{\xi}(\nu),
\]
where, uniformly for \(\xi\in\Xi\) and for \(\nu\) in a fixed
bounded neighborhood of \(1\),
\begin{align*}
\left\|
\mathcal E_{\epsilon,\eta}^{\xi}(\nu)
\right\|
&\leq
C\left(
|\widehat\lambda|\eta^2
+
\epsilon|\widehat\lambda|^2\eta^2
\right)\leq
C'\left(
\frac1{|\log\eta|}
+
\frac{\epsilon}
{\eta^2|\log\eta|^2}
\right).
\end{align*}
Set
\[
\delta_{\epsilon,\eta}
:=
\frac1{|\log\eta|}
+
\frac{\epsilon}
{\eta^2|\log\eta|^2}.
\]
By assumption,
\[
\delta_{\epsilon,\eta}\longrightarrow0.
\]

Consider the principal operator function
\[
\mathcal A_0(\nu):=I-\nu\Pi.
\]
For \(\nu\neq1\),
\[
\mathcal A_0(\nu)^{-1}
=
(I-\Pi)+\frac1{1-\nu}\Pi.
\]
Choose \(M>0\) sufficiently large and let
\[
\Gamma_{\epsilon,\eta}
:=
\left\{
\nu\in\mathbb C:
|\nu-1|=M\delta_{\epsilon,\eta}
\right\}.
\]
For sufficiently small \(\epsilon\) and \(\eta\), this circle is
contained in the neighborhood in which the above expansion holds.
On \(\Gamma_{\epsilon,\eta}\),
\[
\left\|\mathcal A_0(\nu)^{-1}\right\|
\leq
\frac{1}{M\delta_{\epsilon,\eta}},
\qquad
\left\|
\mathcal A_0(\nu)^{-1}
\mathcal E_{\epsilon,\eta}^{\xi}(\nu)
\right\|
\leq
\frac{C'}{M}.
\]
Taking \(M\) sufficiently large makes the second norm strictly less
than one for \(\nu\in\Gamma_{\epsilon,\eta}\), uniformly for
\(\xi\in\Xi\).

By the operator-valued Rouch\'e theorem of Gohberg--Sigal \cite[Chapter~1]{ammari2018mathematical},
\(\mathcal A_0\) and
\(\mathcal A_0+\mathcal E_{\epsilon,\eta}^{\xi}\) have the same total
algebraic multiplicity of characteristic values
inside \(\Gamma_{\epsilon,\eta}\). Since \(\mathcal A_0\) has a
characteristic value of algebraic multiplicity two at \(\nu=1\),
there exist two characteristic values \(\nu_1^\xi\) and \(\nu_2^\xi\)
such that \(|\nu_j^\xi-1|\leq M\delta_{\epsilon,\eta}\) for
\(j=1,2\).
Therefore,
\[
\nu_{1,2}^\xi
=
1
+
O\left(\frac1{|\log\eta|}\right)
+
O\left(
\frac{\epsilon}
{\eta^2|\log\eta|^2}
\right),
\]
uniformly for \(\xi\in\Xi\).

Finally, substituting
\(\widehat\lambda=2\pi\nu/(|D|\eta^2|\log\eta|)\), we obtain
\[
\widehat\lambda_{1,2}(\xi)
=
\frac{2\pi}{|D|\eta^2|\log\eta|}
\left[
1
+
O\left(\frac1{|\log\eta|}\right)
+
O\left(
\frac{\epsilon}
{\eta^2|\log\eta|^2}
\right)
\right].
\]
Since conjugation by the unitary operator \(U_\eta\) preserves
characteristic values and their algebraic multiplicities, the same
conclusion holds for
\(\widehat{\mathcal A}
(\widehat\lambda,\xi;\epsilon,\eta)\).
\end{proof}
\begin{prop}[Two characteristic values for a general quasi-momentum]
\label{prop:general-xi-effective}
Let \(\Xi\Subset\mathbb R^2\) satisfy
\(\inf_{\xi\in\Xi}\inf_{k\in\Lambda^*}|k+\xi|>0\), and set
\[
L:=|\log\eta|,
\qquad
\delta_{\epsilon,\eta}
:=
\frac1L+\frac{\epsilon}{\eta^2L^2}.
\]
Assume that \(\eta\to0\) and \(\epsilon/(\eta^2L)\to0\).
Then, uniformly for \(\xi\in\Xi\), the two characteristic values given by
Theorem~\ref{thm:two_characteristic_values}, counted with algebraic
multiplicity and labeled so that \(\widehat\lambda_-(\xi)\leq
\widehat\lambda_+(\xi)\), satisfy
\begin{equation}\label{equ:general-xi-refined-eigenvalues}
\widehat\lambda_\pm(\xi)
=
\frac{2\pi}{|D|\eta^2L}
\left[
1-\frac{2\pi\mu_\mp(\xi)}{|D|L}
+
O\left(
\frac1{L^2}
+\frac{\epsilon}{\eta^2L^2}
\right)
\right],
\end{equation}
where
\begin{equation}\label{equ:mu-pm-xi}
\mu_\pm(\xi)
=\kappa_\xi\pm|g(\xi)|,
\end{equation}
and
\begin{align}
\kappa_\xi
&:=
\frac1{|D|}
\int_D\int_D
\left[
-\frac1{2\pi}\log|x-y|+R_\xi
\right]dx\,dy,
\label{equ:kappa-xi}
\\
g(\xi)
&:=|D|G_0^\xi(z_1,z_2).
\label{equ:g-xi}
\end{align}
In particular, \(\kappa_{-\xi}=\kappa_\xi\).
Moreover, the associated normalized characteristic functions are,
after applying \(U_\eta\), within
\(O(\delta_{\epsilon,\eta})\) of the two-dimensional constant-mode
space in \(L^2(D)\oplus L^2(D)\).
\end{prop}

\begin{proof}
Let
\[
u_0:=|D|^{-1/2}\mathbf 1_D,
\qquad
e_1:=\begin{pmatrix}u_0\\0\end{pmatrix},
\qquad
e_2:=\begin{pmatrix}0\\u_0\end{pmatrix},
\]
and let \(\Pi\) be the orthogonal projection onto
\(\mathcal E_0:=\operatorname{span}\{e_1,e_2\}\).
Set \(Q:=I-\Pi\).  By Theorem~\ref{thm:two_characteristic_values}, the
two characteristic values under consideration satisfy
\(\widehat\lambda\eta^2=O(L^{-1})\).

Following \eqref{equ:LS_operator_matrix}, define the scaled
two-inclusion integral operator by
\[
\mathbb T_\eta^\xi(\epsilon\widehat\lambda)
:=
\begin{pmatrix}
\mathcal T_D^{\xi,\epsilon\widehat\lambda,\eta}
&\mathcal R_+
\\
\mathcal R_-
&\mathcal T_D^{-\xi,\epsilon\widehat\lambda,\eta}
\end{pmatrix}.
\]
Thus the conjugated characteristic operator is
\(U_\eta\widehat{\mathcal A}
(\widehat\lambda,\xi;\epsilon,\eta)U_\eta^{-1}
=I-\widehat\lambda\eta^2\mathbb T_\eta^\xi(\epsilon\widehat\lambda)\).

The diagonal and off-diagonal kernel expansions give
\begin{equation}\label{equ:T-effective-simple}
\mathbb T_\eta^\xi(\epsilon\widehat\lambda)
=
\frac{|D|L}{2\pi}\Pi+\mathbb B(\xi)
+
O\left(
\eta+\frac{\epsilon}{\eta^2L}
\right),
\end{equation}
where
\[
\mathbb B(\xi)
:=
\begin{pmatrix}
\mathcal K_\xi
&G_0^\xi(z_1,z_2)\mathcal P
\\
G_0^\xi(z_2,z_1)\mathcal P
&\mathcal K_\xi
\end{pmatrix}.
\]
Consequently, in the basis \(\{e_1,e_2\}\),
\begin{equation}\label{equ:B-xi}
\Pi\mathbb B(\xi)\Pi
=
\mathbb K(\xi)
:=
\begin{pmatrix}
\kappa_\xi&g(\xi)\\
\overline{g(\xi)}&\kappa_\xi
\end{pmatrix}.
\end{equation}

Write the characteristic operator with respect to
\(\mathcal E_0\oplus\operatorname{Ran}Q\) as
\[
\mathbb A(\widehat\lambda,\xi)
=I-\widehat\lambda\eta^2\mathbb T_\eta^\xi(\epsilon\widehat\lambda)
=
\begin{pmatrix}
A_{\Pi\Pi}&A_{\Pi Q}\\
A_{Q\Pi}&A_{QQ}
\end{pmatrix}.
\]
Since \(\widehat\lambda\eta^2=O(L^{-1})\), formula
\eqref{equ:T-effective-simple} implies
\[
A_{QQ}=Q+O(L^{-1}),
\qquad
A_{QQ}^{-1}=Q+O(L^{-1}).
\]
Thus \(A_{QQ}\) is invertible on \(\operatorname{Ran}Q\).

If \(f=f_\Pi+f_Q\) satisfies \(\mathbb Af=0\), the \(Q\)-equation gives
\(f_Q=-A_{QQ}^{-1}A_{Q\Pi}f_\Pi\). Substitution into the \(\Pi\)-equation
gives the Schur-complement equation
\begin{equation}\label{equ:Feshbach-equation}
\left(
A_{\Pi\Pi}-A_{\Pi Q}A_{QQ}^{-1}A_{Q\Pi}
\right)f_\Pi=0.
\end{equation}
Conversely, every solution of \eqref{equ:Feshbach-equation} gives a
solution of the full equation through the preceding formula for
\(f_Q\).

The off-diagonal blocks are \(O(L^{-1})\), so
\(A_{\Pi Q}A_{QQ}^{-1}A_{Q\Pi}=O(L^{-2})\).
Using \eqref{equ:T-effective-simple} and \eqref{equ:B-xi}, the
two-dimensional equation \eqref{equ:Feshbach-equation} becomes
\begin{equation}\label{equ:Feshbach-expansion}
\det\left[
\left(
1-\frac{|D|L}{2\pi}\widehat\lambda\eta^2
\right)I_2
-\widehat\lambda\eta^2\mathbb K(\xi)
+
O\left(
\frac1{L^2}
+\frac{\epsilon}{\eta^2L^2}
\right)
\right]=0.
\end{equation}

The eigenvalues of \(\mathbb K(\xi)\) are \(\mu_\pm(\xi)\) in
\eqref{equ:mu-pm-xi}.
For real \(\widehat\lambda\) in the above spectral window, the Schur
complement is Hermitian and strictly decreasing in \(\widehat\lambda\).
Indeed, the Fourier representation gives
\(\partial_{\widehat\lambda}\widehat{\mathcal A}
=-\mathcal T-\epsilon\widehat\lambda\partial_\omega\mathcal T<0\),
with \(\mathcal T=\mathcal T_{D_\eta}^{\xi,\epsilon\widehat\lambda}\),
and differentiation of the Schur complement preserves this inequality.
The min--max principle and the two-root count therefore pair its zeros
with the ordered eigenvalues of the leading matrix, including when
\(\mu_+=\mu_-\). The matrix perturbation bound in
\eqref{equ:Feshbach-expansion} gives
\[
1-\widehat\lambda_\pm(\xi)\eta^2
\left(
\frac{|D|L}{2\pi}+\mu_\mp(\xi)
\right)
=
O\left(
\frac1{L^2}
+\frac{\epsilon}{\eta^2L^2}
\right).
\]
Solving for \(\widehat\lambda_\pm(\xi)\) and expanding the denominator
gives \eqref{equ:general-xi-refined-eigenvalues}.

Finally, the preceding formula for \(f_Q\) implies
\(\|f_Q\|\le C\delta_{\epsilon,\eta}\|f_\Pi\|\), which proves the
characteristic-function estimate.
\end{proof}

\subsection{Dirac point at \texorpdfstring{\(K\)}{K}}
In this subsection, we investigate the Dirac point structure at $\xi=K$. Set \(\kappa=e^{2\pi i/3}\) and define the spatial rotation
\begin{equation}
\rho(x):=q+R(x-q),
\label{equ:physical-rotation}
\end{equation}
where \(q\) is given by \eqref{eq:q_rotation_center}.
The rotation \(u\mapsto u\circ\rho^{-1}\) of \(K\)-quasiperiodic
functions restricts to the operator
\(R_\eta:L^2(D_\eta)\to L^2(D_\eta)\) given by
\[
( R_\eta f)(x)
:=
\begin{cases}
\kappa\,f\!\left(z_1+R^{-1}(x-z_1)\right),
&x\in D_{\eta,1},\\[2mm]
\kappa^2\,f\!\left(z_2+R^{-1}(x-z_2)\right),
&x\in D_{\eta,2}.
\end{cases}
\]
Indeed,
\[
\rho^{-1}z_1=z_1-2\mathbf e_1,\qquad
\rho^{-1}z_2=z_2-\mathbf e_1,
\]
so the phases are \(e^{-2iK\cdot\mathbf e_1}=\kappa\) and
\(e^{-iK\cdot\mathbf e_1}=\kappa^2\), respectively.
We also define the antiunitary inversion-conjugation operator
\( P_\eta:L^2(D_\eta)\to L^2(D_\eta)\) by
\[
( P_\eta f)(x):=\overline{f(-x)},
\qquad x\in D_\eta.
\]
The operator \(R_\eta\) is unitary, and
\[
R_\eta^3=I,\qquad P_\eta^2=I,\qquad
R_\eta P_\eta=P_\eta R_\eta.
\]

The Fourier representation \eqref{equ:Gfourier} gives
\[
G_\lambda^K(\rho x,\rho y)=G_\lambda^K(x,y),\qquad
\overline{G_\lambda^K(-x,-y)}=G_{\overline\lambda}^K(x,y).
\]
Since the periodic inclusion array is invariant under these spatial
symmetries, restriction of the resolvent identities to \(D_\eta\)
yields
\begin{align*}
R_\eta
\widehat{\mathcal A}
(\widehat\lambda,K;\epsilon,\eta)
&=
\widehat{\mathcal A}
(\widehat\lambda,K;\epsilon,\eta)
R_\eta,\\
P_\eta
\widehat{\mathcal A}
(\widehat\lambda,K;\epsilon,\eta)
&=
\widehat{\mathcal A}
(\overline{\widehat\lambda},K;\epsilon,\eta)
P_\eta.
\end{align*}
In particular, \(P_\eta\) commutes with the characteristic operator
when \(\widehat\lambda\) is real.

Since \(R_\eta\) is unitary and \(R_\eta^3=I\), the space \(L^2(D_\eta)\) admits the orthogonal
decomposition
\[
L^2(D_\eta)
=
L_0\oplus L_1\oplus L_2,
\]
where
\[
L_j
:=
\left\{
f\in L^2(D_\eta):
R_\eta f=\kappa^j f
\right\},
\qquad
j=0,1,2.
\]
Each of these subspaces is invariant under
\(\widehat{\mathcal A}
(\widehat\lambda,K;\epsilon,\eta)\).

Moreover, the commutation of \(R_\eta\) and the antiunitary \(P_\eta\) gives
\[
P_\eta(L_1)=L_2,
\qquad
P_\eta(L_2)=L_1,
\qquad
P_\eta(L_0)=L_0.
\]

Notice that
\[
R_\eta\chi_{D_{\eta,1}}
=
\kappa\chi_{D_{\eta,1}},
\qquad
R_\eta\chi_{D_{\eta,2}}
=
\kappa^2\chi_{D_{\eta,2}},
\]
so that
\[
\chi_{D_{\eta,1}}\in L_1,
\qquad
\chi_{D_{\eta,2}}\in L_2.
\]

Repeating the proof of Theorem~\ref{thm:two_characteristic_values}
within the invariant subspaces \(L_1\) and \(L_2\), respectively,
one obtains exactly one characteristic value, counted with algebraic
multiplicity, in each subspace near
\[
\frac{2\pi}{|D|\eta^2|\log\eta|}.
\]

Since \(\epsilon=1/\tau>0\), Theorem~\ref{thm:LSequation} and
self-adjointness of \(P_K\) imply that these characteristic values
are real.  Thus the symmetry \(P_\eta\) implies that the two
characteristic values coincide. Indeed, if \(\phi_1(\eta,\epsilon)\in L_1\) with
\[
\widehat{\mathcal A}
(\widehat\lambda,K;\epsilon,\eta)\phi_1(\eta,\epsilon)=0,
\]
then
\[
\widehat{\mathcal A}
(\widehat\lambda,K;\epsilon,\eta)
P_\eta\phi_1(\eta,\epsilon)
=0,
\]
and
\[
P_\eta\phi_1(\eta,\epsilon)\in L_2.
\]
Since there is only one characteristic value in \(L_2\) near
\(\frac{2\pi}{|D|\eta^2|\log\eta|}\), it follows that
\[
\widehat\lambda_1(K)
=
\widehat\lambda_2(K).
\]
Each sector kernel is one-dimensional.  Since
Theorem~\ref{thm:two_characteristic_values} counts only two
characteristic values in this window, there is no additional
kernel in \(L_0\).  Consequently, the common characteristic value
has geometric multiplicity exactly two.

\begin{coro}[The double characteristic value at \(K\)]
\label{coro:refined-K-double}
Under the assumptions of Proposition~\ref{prop:general-xi-effective},
with \(K\in\Xi\), the two characteristic values coincide at \(\xi=K\).  Their common
value satisfies
\begin{equation}\label{equ:lambda0-refined-K}
\widehat\lambda_0^\epsilon(\eta)
=
\frac{2\pi}{|D|\eta^2L}
\left[
1-\frac{2\pi\kappa_K}{|D|L}
+
O\left(
\frac1{L^2}
+\frac{\epsilon}{\eta^2L^2}
\right)
\right],
\end{equation}
where \(\kappa_K\) is defined by \eqref{equ:kappa-xi} at \(\xi=K\).
Its geometric multiplicity is two.
\end{coro}

\begin{proof}
The preceding symmetry argument already proves the exact coincidence
and geometric multiplicity two.  In the rescaled space, the constant
modes satisfy \(e_j\in U_\eta L_j\), \(j=1,2\).
The rescaled rotation \(U_\eta R_\eta U_\eta^{-1}\) is independent
of \(\eta\) and commutes with \(\Pi\).  Taking
\(\eta\to0\) and \(\epsilon/(\eta^2L)\to0\) in the commutator
of \eqref{equ:T-effective-simple} shows that \(\mathbb B(K)\)
also preserves its eigenspaces.
  Therefore
\(\langle e_1,\mathbb B(K)e_2\rangle=0\).  By
\eqref{equ:B-xi}, this gives
\begin{equation}\label{equ:gK-zero}
g(K)=|D|G_0^K(z_1,z_2)=0.
\end{equation}
By \(R_{-\xi}=R_\xi\), we have
\(\kappa_K=\kappa_{-K}\).
Consequently,
\[
\mathbb K(K)=\kappa_K I_2,
\qquad
\mu_+(K)=\mu_-(K)=\kappa_K.
\]
Formula \eqref{equ:lambda0-refined-K} follows from
\eqref{equ:general-xi-refined-eigenvalues} evaluated at \(\xi=K\).
\end{proof}

\begin{thm}[Local eigenvalue asymptotics near \(K\)]
\label{thm:refined-Dirac-point}
Assume the hypotheses of Proposition~\ref{prop:general-xi-effective},
with \(\Xi\) containing a neighborhood of \(K\).
The two characteristic values coincide at \(K\), with common
value \(\widehat\lambda_0^\epsilon(\eta)\) as in
Corollary~\ref{coro:refined-K-double}.
There exists a fixed neighborhood
\(\mathcal U_K\Subset\operatorname{int}\Xi\) of \(K\) such that,
with the same upper/lower labeling as above, uniformly
for \(K+h\in\mathcal U_K\) in the stated parameter regime,
\begin{align}
\widehat\lambda_\pm(K+h)
={}&
\widehat\lambda_0^\epsilon(\eta)
\pm
\frac{4\pi^2|c|}{|D|\eta^2L^2}
|h|
+O\left(
\frac{|h|^2}{\eta^2L^2}
+
\frac1{\eta^2L^3}
+
\frac{\epsilon}{\eta^4L^3}
\right),
\label{equ:refined-Dirac-cone}
\end{align}
where
\[
\left.
\nabla_\xi G_0^\xi
\right|_{\xi=K}(z_1,z_2)
=
c
\begin{pmatrix}1\\i\end{pmatrix},
\qquad
c\ne0.
\]
\end{thm}

\begin{proof}
Threefold rotational symmetry gives \(\kappa_{R\xi}=\kappa_\xi\),
with quasi-momenta understood modulo \(\Lambda^*\). Since
\(RK=K\pmod{\Lambda^*}\), the vector \(\nabla_\xi\kappa_K\) must be
fixed by the rotation \(R\), and hence it vanishes. Using
\(\kappa_{-\xi}=\kappa_\xi\), we obtain
\[
\kappa_{K+h}
\;=\;\kappa_{-K-h}
\;=\;
\kappa_K+O(|h|^2).
\]
Set \(F(\xi):=G_0^\xi(z_1,z_2)\) and \(d:=z_1-z_2\).
Since \(R^{-1}d=d-\mathbf e_1\), rotation covariance and
quasiperiodicity give
\[
 F(R\xi)=e^{-i\xi\cdot\mathbf e_1}F(\xi).
\]
Differentiating at \(K\), using \(F(K)=0\) and
\(RK\equiv K\pmod{\Lambda^*}\), yields
\(R^T\nabla F(K)=\kappa^2\nabla F(K)\).
The corresponding eigenspace is \(\mathbb C(1,i)^T\), proving
the asserted form \(\nabla F(K)=c(1,i)^T\).
On the other hand, \eqref{equ:gK-zero} and Taylor expansion give
\begin{align*}
g(K+h)
&=
|D|\left[
G_0^K(z_1,z_2)
+
h\cdot
\left.
\nabla_\xi G_0^\xi
\right|_{\xi=K}(z_1,z_2)
\right]
+
O(|h|^2)
\\
&=
|D|c(h_1+ih_2)+O(|h|^2).
\end{align*}
Appendix~\ref{app:c-nonzero} shows that \(c\ne0\).  Hence
\[
|g(K+h)|
=
|D||c|\,|h|+O(|h|^2),
\]
and \eqref{equ:mu-pm-xi} yields
\[
\mu_\pm(K+h)
=
\kappa_K
\pm |D||c|\,|h|
+
O(|h|^2).
\]
Substitution into \eqref{equ:general-xi-refined-eigenvalues} gives
the asserted linear term, with Taylor error \(O(|h|^2/(\eta^2L^2))\).
\end{proof}

\begin{remark}[The exact Dirac cone]
\label{rem:exact-Dirac-cone}
The uniform eigenvalue error in Theorem~\ref{thm:refined-Dirac-point}
does not by itself prove a nonzero conical slope.  Instead, differentiating
the exact operator pencil and taking its Schur complement on the
two-dimensional kernel proves that the exact off-diagonal
first-order coefficient is nonzero; see
Appendix~\ref{app:exact-Dirac-cone}.
Under the same hypotheses, for sufficiently small \(\eta\) and
\(\epsilon/(\eta^2L)\), the two exact characteristic values satisfy
\begin{align}
 \widehat\lambda_\pm(K+h)
 ={}&\widehat\lambda_0^\epsilon(\eta)
 \pm\left[
 \frac{4\pi^2|c|}{|D|\eta^2L^2}
 +O\left(\frac1{\eta^2L^3}
          +\frac{\epsilon}{\eta^4L^3}\right)
 \right]|h|+O_{\epsilon,\eta}(|h|^2),\qquad h\to0.
 \label{equ:exact-Dirac-cone}
\end{align}
The slope is strictly positive. 
\end{remark}

\section{Continuation with respect to the inclusion size}
\label{sec:eta-continuation}

We fix the contrast \(\tau>0\) and continue a Dirac point in the
single real parameter \(\eta\).  The starting assumption is one
eigenvalue satisfying the three conditions in
\eqref{equ:definition-good-pair} below. Rellich's theorem then allows us to
follow analytic eigenvalue branches through spectral collisions.
Throughout this section, \(\partial D\) is \(C^2\), and
\((0,\eta_{\mathrm{c}})\) is the interval before the first contact between
distinct inclusions in the periodic array.

\subsection{Fixed-domain forms and symmetries}

Fix a reference size \(\eta_*\in(0,\eta_{\mathrm{c}})\).
Choose a smooth periodic vector field \(V\) equal to
\(x-z_j-\ell\) near \(D_{\eta_*,j}+\ell\).
The prescribed fields respect rotation and inversion, so averaging
over the finite symmetry group on the torus makes \(V\) equivariant
without changing these prescriptions.  For \(\eta\) sufficiently
close to \(\eta_*\), set
\[
 \Phi_\eta(x)=x+\Big(\frac{\eta}{\eta_*}-1\Big)V(x),
 \qquad F_\eta=D\Phi_\eta,\qquad J_\eta=\det F_\eta.
\]
Since \(F_\eta\) is uniformly close to \(I\), \(\Phi_\eta\) is a
diffeomorphism of the torus, with a lift satisfying
\(\Phi_\eta(x+\ell)=\Phi_\eta(x)+\ell\).  On each reference inclusion
it is the required dilation, hence
\(\Phi_\eta(D_{\eta_*,j})=D_{\eta,j}\).

On periodic functions, use the phase-corrected pullback
\[
 (\mathscr U_{p,\eta}u)(x)
 =e^{ip\cdot(\Phi_\eta(x)-x)}
 J_\eta(x)^{1/2}u(\Phi_\eta(x)).
\]
Since \(\Phi_\eta(x)-x\) is periodic, this map is unitary from
\(L^2_{\mathrm{per}}(Y,n_\eta\,dx)\) to
\(L^2_{\mathrm{per}}(Y,n_{\eta_*}\,dx)\) and maps
\(H^1_{\mathrm{per}}(Y)\) onto itself for real \(p,\eta\).
It is obtained by pulling back the physical Bloch function
\(e^{ip\cdot x}u(x)\) and then removing \(e^{ip\cdot x}\)
in the reference coordinates.  Changing variables gives the
equivalent eigenvalue problem
\begin{equation}
 a_{p,\eta}(v,u)=\lambda b(v,u)
 \quad\text{for all }v\in H^1_{\mathrm{per}}(Y),
 \label{equ:continued-weak-problem}
\end{equation}
where
\begin{align}
 a_{p,\eta}(u,v)
 &:=\int_Y J_\eta\,
 \overline{\left[
 F_\eta^{-T}(\nabla+ip)(J_\eta^{-1/2}u)
 \right]}
 \cdot\left[
 F_\eta^{-T}(\nabla+ip)(J_\eta^{-1/2}v)
 \right]\,dx,
 \label{equ:continued-a-form}\\
 b(u,v)&:=\int_Y n_{\eta_*}\,\overline u v\,dx.
 \label{equ:continued-b-form}
\end{align}
In particular, the mass form is independent of \(\eta\).
The coefficients in these forms are real analytic in \((p,\eta)\):
\(F_\eta\) is affine in \(\eta\), and its inverse and
\(J_\eta^{\pm1/2}\) are analytic near \(\eta_*\).
Moreover, \(a_{p,\eta}+b\) is locally uniformly coercive on
\(H^1_{\mathrm{per}}(Y)\).  The closed-form perturbation theorem
\cite[Chapter VII, \S4]{kato1995perturbation} therefore gives a
locally analytic family \(H(p,\eta)\) on
\(L^2_{\mathrm{per}}(Y,n_{\eta_*}\,dx)\),
self-adjoint for real parameters.  Its resolvent is compact because
\(H^1_{\mathrm{per}}(Y)\) embeds compactly into
\(L^2_{\mathrm{per}}(Y,n_{\eta_*}\,dx)\).
Its eigenvalues are precisely the physical Bloch eigenvalues at
quasi-momentum \(p\).

We use the same spatial rotation \(\rho\) from
\eqref{equ:physical-rotation} on the fixed space and set
\[
 (\mathscr Ru)(x)
 =e^{iK\cdot(\rho^{-1}x-x)}u(\rho^{-1}x),\qquad
 (\mathscr Pu)(x)=\overline{u(-x)}.
\]
These are the periodic representations of the \(K\)-quasiperiodic
symmetries under multiplication by \(e^{-iK\cdot x}\).
The phase in \(\mathscr R\) is periodic since
\(RK-K\in\Lambda^*\).  Thus \(\mathscr R\) is unitary,
\(\mathscr P\) is antiunitary, and the corresponding identities give
\[
 \mathscr R^3=\mathscr P^2=I,\qquad
 \mathscr R\mathscr P=\mathscr P\mathscr R.
\]
Writing \(h=p-K\), equivariance of \(\Phi_\eta\) gives
\[
\mathscr R\mathscr U_{K+h,\eta}
=\mathscr U_{K+Rh,\eta}\mathscr R,\qquad
\mathscr P\mathscr U_{p,\eta}=\mathscr U_{p,\eta}\mathscr P.
\]
The transformed forms therefore satisfy
\[
 H(K+Rh,\eta)=\mathscr R H(K+h,\eta)\mathscr R^{-1},
 \qquad H(p,\eta)\mathscr P=\mathscr P H(p,\eta).
\]
Using \(\kappa=e^{2\pi i/3}\) as before, write
\begin{equation}
 \mathcal H_j
 =\{u\in L^2_{\mathrm{per}}(Y,n_{\eta_*}\,dx):
                  \mathscr Ru=\kappa^j u\},\qquad j=0,1,2.
 \label{equ:rotation-sector-projections}
\end{equation}
These are mutually orthogonal reducing subspaces for \(H(K,\eta)\),
and \(\mathscr P\mathcal H_{1}=\mathcal H_{2}\).
We write \(H_j(\eta)=H(K,\eta)|_{\mathcal H_j}\), with form
domain \(H^1_{\mathrm{per}}(Y)\cap\mathcal H_j\).

\subsection{A criterion for a Dirac point}

For \(u\in H^1_{\mathrm{per}}(Y)\), set
\[
 \vartheta_\eta(u)
 :=\left.\partial_{p_1}a_{p,\eta}(u,\mathscr Pu)\right|_{p=K}.
\]
The derivatives of \(a_{p,\eta}\) in \(p\) are bounded forms on
\(H^1_{\mathrm{per}}(Y)\) and depend real analytically on \(\eta\).
Fix \(\eta\).  We show that a simple sector-\(1\) eigenvalue
absent from sector \(0\) gives a Dirac point whenever
\(\vartheta_\eta(u)\ne0\).

Suppose \(z\) is a simple eigenvalue of \(H_1(\eta)\), with
\(b\)-normalized eigenfunction \(u\), and
\(z\notin\sigma(H_0(\eta))\).
Since \(\mathscr P\) intertwines sectors \(1\) and \(2\),
\(z\) has multiplicity exactly two for \(H(K,\eta)\), with eigenspace
\(\operatorname{span}\{u,\mathscr Pu\}\).
Let \(\Gamma\) be a positively
oriented circle enclosing \(z\) and no other spectral point of
\(H(K,\eta)\). Define the Riesz projection
\[
\Pi(p):=\frac{1}{2\pi i}\oint_\Gamma
(\zeta-H(p,\eta))^{-1}\,d\zeta.
\]
The projected vectors \(\Pi(p)u\) and \(\Pi(p)\mathscr Pu\)
remain linearly independent near \(K\), so
Gram--Schmidt orthonormalization with respect to the inner product
\(b\) gives an \(H^1_{\mathrm{per}}(Y)\)-analytic basis
\(u(p),v(p)\) satisfying
\[
 u(K)=u,\quad v(K)=\mathscr Pu,\qquad
 b(u(p),u(p))=b(v(p),v(p))=1,\quad b(u(p),v(p))=0.
\]

In this orthonormal basis the restricted operator is represented by
\[
 \begin{pmatrix}
 a_{p,\eta}(u(p),u(p))&a_{p,\eta}(u(p),v(p))\\
 a_{p,\eta}(v(p),u(p))&a_{p,\eta}(v(p),v(p))
 \end{pmatrix}.
\]
Its eigenvalues are exactly the two Bloch eigenvalues near \(z\).
It is Hermitian and real analytic in \(p\), because both the form
and the basis are real analytic in the form norm.  At \(K\), the
eigenvalue equation makes this matrix equal to \(zI_2\).

For \(r=1,2\), the product rule for the upper-right entry gives
\[
 \begin{aligned}
 \partial_{p_r}\bigl[a_{p,\eta}(u(p),v(p))\bigr]
 ={}&(\partial_{p_r}a_{p,\eta})(u(p),v(p))\\
 &+a_{p,\eta}(\partial_{p_r}u(p),v(p))
  +a_{p,\eta}(u(p),\partial_{p_r}v(p)).
 \end{aligned}
\]
At \(K\), both basis vectors have eigenvalue \(z\).
The weak eigenvalue equation, together with Hermitian symmetry,
therefore turns the last two terms into
\[
 \begin{aligned}
 &\left.
 \bigl[a_{K,\eta}(\partial_{p_r}u,v)
       +a_{K,\eta}(u,\partial_{p_r}v)\bigr]\right|_{p=K}\\
 &\qquad=z\left.
 \bigl[b(\partial_{p_r}u,v)+b(u,\partial_{p_r}v)\bigr]
 \right|_{p=K}
 =z\left.\partial_{p_r}b(u(p),v(p))\right|_{p=K}=0.
 \end{aligned}
\]
The same cancellation applies to the diagonal entries, using the
constant norms instead of orthogonality.  Hence all first-order
entries are obtained by differentiating only the form, keeping
the basis fixed at \(\{u,\mathscr Pu\}\).  In particular,
\[
 \left.\partial_{p_1}
       \bigl[a_{p,\eta}(u(p),v(p))\bigr]\right|_{p=K}
 =\left.(\partial_{p_1}a_{p,\eta})(u,\mathscr Pu)\right|_{p=K}
 =\vartheta_\eta(u).
\]

Write \(h=p-K\).  Rotation covariance makes the diagonal
first-order entries invariant linear functions of \(h\), so they
vanish.  The upper-right first-order entry transforms by
\(\kappa^{-1}\), as does \(h_1+ih_2\) under the clockwise
rotation \(R\).  Its \(h_1\)-coefficient, computed above, is
\(\vartheta_\eta(u)\).  Hermitian symmetry determines the
lower-left entry.  The reduced matrix is therefore
\begin{equation}
 zI_2+
 \begin{pmatrix}
 0&\vartheta_\eta(u)(h_1+ih_2)\\
 \overline{\vartheta_\eta(u)}(h_1-ih_2)&0
 \end{pmatrix}
 +O_\eta(|h|^2).
 \label{equ:continued-Dirac-matrix}
\end{equation}
Consequently,
\begin{equation}
 \lambda_\pm(p,\eta)
 =z\pm|\vartheta_\eta(u)|\,|p-K|+O_\eta(|p-K|^2).
 \label{equ:continued-cone}
\end{equation}
If \(\vartheta_\eta(u)\ne0\), these two bands form a nondegenerate
Dirac cone with slope \(|\vartheta_\eta(u)|\).
The criterion does not depend on the normalized eigenfunction
chosen in sector \(1\): replacing \(u\) by \(e^{it}u\) multiplies
\(\vartheta_\eta(u)\) by \(e^{-2it}\), leaving its modulus unchanged.

We call \((z,\eta)\) \emph{good} if
\begin{equation}
 \begin{cases}
 z\text{ is a simple eigenvalue of }H_{1}(\eta),\\
 z\notin\sigma(H_{0}(\eta)),\\
 \vartheta_\eta(u)\ne0
   \text{ for a normalized }H_{1}(\eta)\text{-eigenfunction }u.
 \end{cases}
 \label{equ:definition-good-pair}
\end{equation}
Every good pair is thus a nondegenerate Dirac point.

\subsection{Rellich branches and analytic detection of good pairs}

We record the local analytic information needed for
\(\vartheta_\eta\).  Fix
\(\eta_0\in(0,\eta_{\mathrm{c}})\) and
\(z_c\in\sigma(H_{1}(\eta_0))\).
Choose a circle \(\Gamma\) about \(z_c\) enclosing no other distinct
eigenvalue of \(H_{1}(\eta_0)\) or \(H_{0}(\eta_0)\).
For \(\eta\) near \(\eta_0\), the associated sector Riesz projections
are analytic and have constant finite ranks \(N\ge1\) and \(M\ge0\).
They are analytic as maps from \(L^2_{\mathrm{per}}(Y,n_{\eta_*}\,dx)\)
to \(H^1_{\mathrm{per}}(Y)\).

For each sector spectral subspace, project a basis at \(\eta_0\)
onto the nearby spectral subspaces and apply Gram--Schmidt
orthonormalization with respect to the inner product \(b\).
This gives an orthonormal basis depending real analytically on
\(\eta\) in \(H^1_{\mathrm{per}}(Y)\).
In these bases the restricted operators are real-analytic matrices,
Hermitian for real \(\eta\) by self-adjointness of the operators.
These properties allow us to apply the finite-dimensional Rellich theorem
\cite[Chapter II, \S6]{kato1995perturbation}, which supplies branches
\[
 \lambda_1(\eta),\ldots,\lambda_N(\eta),\qquad
 \mu_1(\eta),\ldots,\mu_M(\eta),
\]
listing the enclosed sector-\(1\) and sector-\(0\) eigenvalues with
multiplicity, and normalized sector-\(1\) eigenfunctions
\(u_1(\eta),\ldots,u_N(\eta)\).
All eigenvalues are real analytic across \(\eta_0\), and the
eigenfunctions are real analytic with values in \(H^1_{\mathrm{per}}(Y)\).
Rellich's labeling is analytic, not necessarily increasing in energy.

For each \(j\), define the complex-valued real-analytic function
\begin{equation}
 g_j(\eta)=
 \vartheta_\eta\bigl(u_j(\eta)\bigr)
 \prod_{\substack{k=1\\k\ne j}}^N
       \bigl(\lambda_j(\eta)-\lambda_k(\eta)\bigr)
 \prod_{\ell=1}^M
       \bigl(\lambda_j(\eta)-\mu_\ell(\eta)\bigr),
 \label{equ:rellich-good-function}
\end{equation}
where an empty product is \(1\).  The first factor is real analytic
because \(\left.\partial_{p_1}a_{p,\eta}\right|_{p=K}\)
is an analytic bounded form on
\(H^1_{\mathrm{per}}(Y)\).
After shrinking the neighborhood, the listed eigenvalues contain
all spectrum close to \(z_c\).  Hence, for a branch with
\(\lambda_j(\eta_0)=z_c\),
\begin{equation}
 g_j(\eta)\ne0
 \quad\Longleftrightarrow\quad
 (\lambda_j(\eta),\eta)\text{ is good}.
 \label{equ:rellich-good-equivalence}
\end{equation}
In particular, the product detects simplicity as well as the other
two conditions; simplicity is not assumed across a collision.

If \(g_j\) is nonzero at one point of this neighborhood, it is not
identically zero, so its zeros are isolated.  In particular, it is
nonzero for \(0<|\eta-\eta_0|<\delta\), with sufficiently small
\(\delta>0\).  This is the mechanism for continuation through a
critical parameter.

\subsection{Global continuation}

\begin{thm}[Continuation from one good pair]
\label{thm:eta-global-continuation}
Fix \(\tau>0\), and suppose \((z_s,\eta_s)\) is good for some
\(\eta_s\in(0,\eta_{\mathrm{c}})\).
There is a real-analytic function
\(\lambda:(0,\eta_{\mathrm{c}})\to\mathbb R\), with
\(\lambda(\eta_s)=z_s\), which is a sector-\(1\) eigenvalue for every
\(\eta\).  The set
\[
 \mathcal C_\tau
 :=\{\eta\in(0,\eta_{\mathrm{c}}):(\lambda(\eta),\eta)
                                  \text{ is not good}\}
\]
is relatively closed and discrete. The ordered band indices may change at spectral crossings.
\end{thm}

\begin{proof}
\textit{Step 1: start a local analytic branch.}
Since \((z_s,\eta_s)\) is good, \(z_s\) is a simple eigenvalue
of \(H_1(\eta_s)\).  The local analytic theory gives a
real-analytic eigenvalue branch \(\lambda(\eta)\) with
\(\lambda(\eta_s)=z_s\), together with a normalized analytic
eigenfunction.  Spectral isolation and continuity of the conical
coefficient ensure, after shrinking the neighborhood, that this
branch is good throughout it.

\medskip\noindent
\textit{Step 2: control the eigenvalue along a continuation.}
Cover a compact interval \(I\Subset(0,\eta_{\mathrm{c}})\) by finitely
many smaller neighborhoods on which the fixed-reference
constructions are valid.  Boundedness of the form derivatives
and uniform coercivity on these neighborhoods give a common
constant \(C_I\) such that, in each representation,
\[
 |\partial_\eta a_{K,\eta}(u,u)|
 \le C_I\bigl(a_{K,\eta}(u,u)+b(u,u)\bigr),
 \qquad u\in H^1_{\mathrm{per}}(Y).
\]
For a local normalized Rellich eigenpair, differentiation gives
\[
 \lambda'(\eta)=\partial_\eta a_{K,\eta}(u,u),\qquad
 |\lambda'(\eta)|\le C_I(1+\lambda(\eta)).
\]
Indeed, the terms involving \(u'\) cancel by the eigenvalue
equation and \(\partial_\eta b(u(\eta),u(\eta))=0\), since
\(b\) is fixed within each representation.  
Since \(\lambda\ge0\), integration yields
\begin{equation}
 1+\lambda(\eta)
 \le (1+\lambda(\eta_0))e^{C_I|\eta-\eta_0|}
 \label{equ:continued-eigenvalue-bound}
\end{equation}
whenever the branch is defined between \(\eta_0\) and \(\eta\)
in \(I\).  Thus at an interior endpoint the eigenvalue remains
bounded, and the derivative bound gives a finite limit.

\medskip\noindent
\textit{Step 3: extend across an interior endpoint.}
Suppose the initial branch has been continued to an open interval
containing \(\eta_s\), remaining good except on a locally discrete
set, and let \(r<\eta_{\mathrm{c}}\) be its right endpoint.
Choose a compact admissible interval containing \([\eta_s,r]\)
in its interior.  Step 2 bounds both \(\lambda\) and
\(\lambda'\) up to \(r\), so
\[
 z_c:=\lim_{\eta\uparrow r}\lambda(\eta)<\infty.
\]
We claim that \(z_c\in\sigma(H_1(r))\).
Otherwise, choose \(\delta_0>0\) such that
\([z_c-\delta_0,z_c+\delta_0]\) is disjoint from
\(\sigma(H_1(r))\).  In a fixed reference representation near
\(r\), norm-resolvent continuity ensures that this interval
remains disjoint from \(\sigma(H_1(\eta))\) for \(\eta\) close
to \(r\).  But \(\lambda(\eta)\to z_c\), so the eigenvalue
\(\lambda(\eta)\in\sigma(H_1(\eta))\) lies in this interval
for \(\eta<r\) sufficiently close to \(r\), a contradiction.

Apply the local Rellich construction at \((z_c,r)\), obtaining
branches \(\lambda_1,\ldots,\lambda_N\) on both sides of \(r\)
that list all sector-\(1\) eigenvalues near \(z_c\).
Since \(\lambda(\eta)\to z_c\), the incoming branch belongs to
this list on \((r-\delta,r)\).  By real analyticity, it agrees
throughout this overlap with one fixed Rellich branch
\(\lambda_j\).  Use \(\lambda_j\) to define its continuation
at \(r\) and to the right; the agreement on the overlap makes
this an analytic continuation.

There is a good point in the overlap, since the incoming bad set
is locally discrete.  Hence the corresponding \(g_j\) is nonzero
at that point and is not identically zero near \(r\).
Its zeros are isolated, so after shrinking \(\delta\),
\[
 g_j(\eta)\ne0,\qquad 0<|\eta-r|<\delta.
\]
By \eqref{equ:rellich-good-equivalence}, the extended branch is
good on this punctured neighborhood; only \(r\) itself may be
bad.  The same construction applies at a left interior endpoint.

\medskip\noindent
\textit{Step 4: continue over the full parameter interval.}
Consider all analytic continuations of the initial branch to
open intervals containing \(\eta_s\), each good except on a
locally discrete set.  Any two agree near \(\eta_s\), and hence,
by real analyticity, throughout their overlap.  Their union
therefore defines a single analytic branch on a maximal open
interval \(J\subset(0,\eta_{\mathrm{c}})\).  The bad set remains locally
discrete, since each point of \(J\) belongs to one of these
continuation intervals.  If \(J\) had an interior endpoint,
Step 3 would extend the branch beyond it while preserving this
property, contradicting maximality.  Consequently,
\[
 J=(0,\eta_{\mathrm{c}}).
\]

\medskip\noindent
\textit{Step 5: identify the exceptional set.}
Fix any interior parameter \(\eta_0\).  By real analyticity,
\(\lambda\) agrees near \(\eta_0\) with one of the local Rellich
branches \(\lambda_j\).  There is a nearby good point, so the
associated \(g_j\) is not identically zero.
Equation \eqref{equ:rellich-good-equivalence} identifies
\(\mathcal C_\tau\) locally with the zero set of \(g_j\).
It is therefore relatively closed and discrete, and its
intersection with any compact admissible interval is finite.
\end{proof}

For every sufficiently large \(\tau\), one can choose a small
admissible size \(\eta_s\) such that both \(\eta_s\) and
\(1/(\tau\eta_s^2|\log\eta_s|)\) are sufficiently small.
At \(K\), the Lippmann--Schwinger restriction and extension intertwine
the rotation representations. Thus Corollary~\ref{coro:refined-K-double}
gives \(z_s:=\tau^{-1}\widehat\lambda_0^\epsilon(\eta_s)\) as a simple
eigenvalue of \(H_1(\eta_s)\), absent from \(\sigma(H_0(\eta_s))\).
Let \(\widehat v_s>0\) be the exact slope in
Appendix~\ref{app:exact-Dirac-cone}. Since \(\lambda=\tau^{-1}\widehat\lambda\),
\eqref{equ:continued-cone} gives
\( |\vartheta_{\eta_s}(u)|=\tau^{-1}\widehat v_s>0\).
Hence \((z_s,\eta_s)\) is good.
Keeping \(\tau\) fixed, apply
Theorem~\ref{thm:eta-global-continuation} to continue this branch
throughout \((0,\eta_{\mathrm{c}})\).
Consequently, for every sufficiently large \(\tau\), a
nondegenerate Dirac point at \(K\) exists for almost every
admissible size \(\eta\).  More precisely, there is a relatively
closed discrete set \(\mathcal C_\tau\subset(0,\eta_{\mathrm{c}})\)
such that the continued branch is a nondegenerate Dirac point for every
\[
 \eta\in(0,\eta_{\mathrm{c}})\setminus\mathcal C_\tau.
\]

\section{Numerical experiments}
\label{sec:numerical-experiments}

We use the three-petal reference inclusion
\[
 D=\{(r\cos\theta,r\sin\theta):
       0\le r<1+0.3\cos(3\theta),\ 0\le\theta<2\pi\},
 \qquad |D|=1.045\pi.
\]
This gives a concrete geometry of the type shown in
Figure~\ref{fig:two_inclusions}.  It is invariant under rotation
by \(2\pi/3\); the pair \(z_1+\eta D,z_2-\eta D\) also has the
required inversion symmetry.  

The original Bloch problem is solved by the Fourier material-matrix
discretization
\[
 \sum_{G'\in\Lambda^*}
 \left[|G+\xi|^2\delta_{GG'}
 -\lambda\left(\delta_{GG'}+
 \frac{\tau}{|Y|}\int_{D_\eta}e^{-i(G-G')\cdot x}\,dx\right)
 \right]c_{G'}=0,
\]
truncated to \(|G+\xi|\le2\pi M\).  Boundary quadrature evaluates
the material coefficients without pixelating the interfaces.
For very small inclusions we instead pull each component back to the
fixed reference domain \(D\) by \(x=z_1+\eta X\) and
\(x=z_2-\eta X\), respectively, with \(X\in D\). The logarithmic part of
the Green kernel is integrated analytically on the diagonal cells,
and the resulting Nystr\"om matrix has a size independent of \(\eta\).
We solve directly for \(\eta^2\widehat\lambda\), thereby avoiding the
large numerical scale of \(\widehat\lambda\).  For \(0.01\leq\eta\leq0.1\), 
where the Fourier calculation remains feasible, the two discretizations give consistent values for the two lowest scaled eigenvalues.

\subsection{Eigenvalue asymptotics}
\label{sec:numerical-eigenvalue-verification}

We first test Proposition~\ref{prop:general-xi-effective}.  The lowest
two eigenvalues are computed independently of
\eqref{equ:general-xi-refined-eigenvalues} and rescaled by
\(\widehat\lambda=\tau\lambda\).  The coefficients \(\mu_\pm(\xi)\) are computed 
independently from the geometry of the reference inclusion and Ewald evaluations of the Green function.

We use \(\tau\eta^2=10L\), where \(L=|\log\eta|\), and sample
\(\eta\) between \(0.1\) and \(10^{-20}\).  The four momenta are
\[
 \xi\in\left\{K,K+(0.25,0.12)^T,0.3K,
 \frac{\boldsymbol\beta_1+\boldsymbol\beta_2}{4}\right\}.
\]
Figure~\ref{fig:flower-eigenvalue-test}
compares the numerical eigenvalues with the asymptotic prediction.

\begin{figure}[!htbp]
 \centering
 \includegraphics[width=\textwidth]{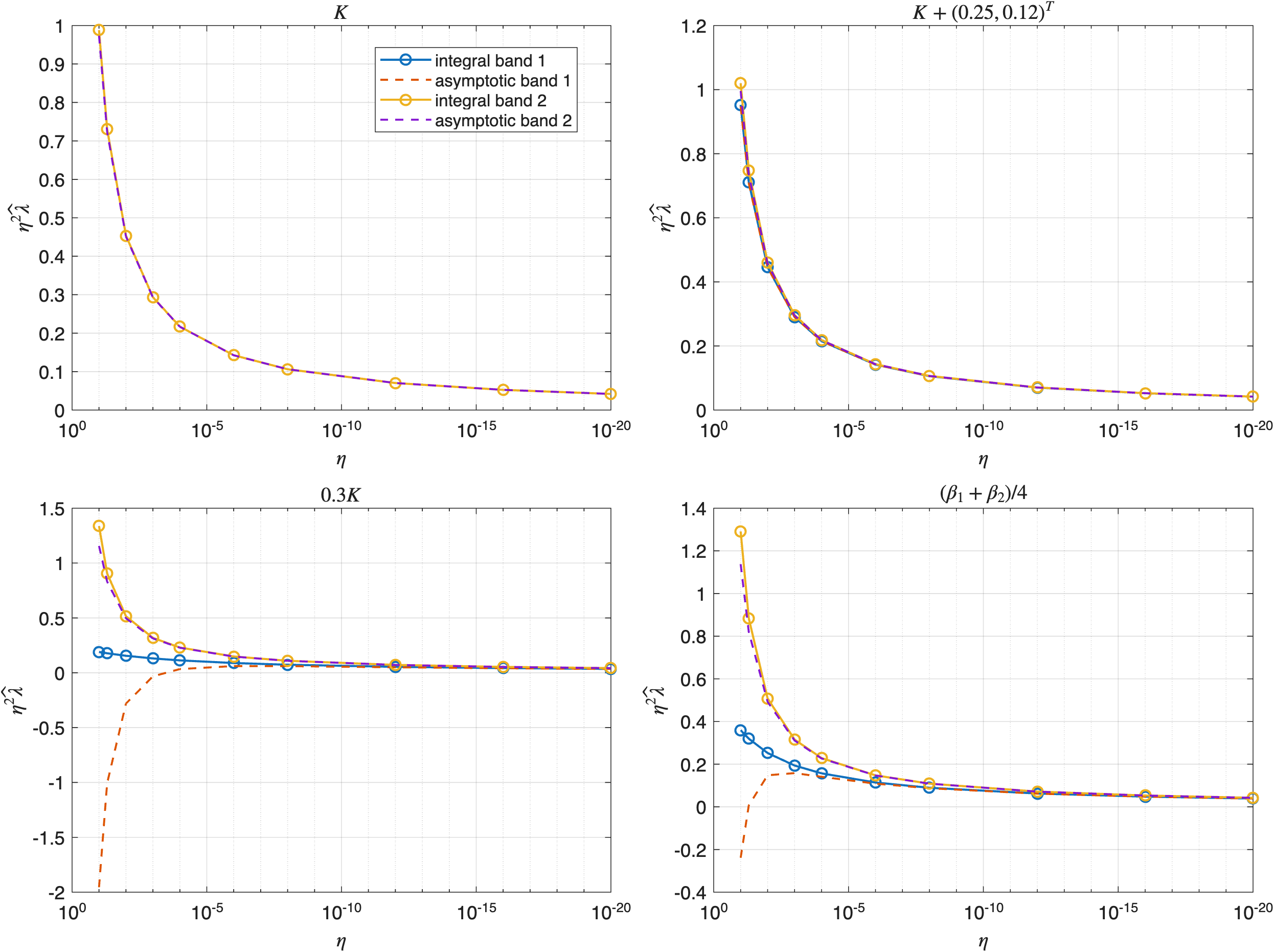}
 \caption{Three-petal inclusions: fixed-domain integral eigenvalues versus
 \eqref{equ:general-xi-refined-eigenvalues} for the two scaled
 eigenvalues along \(\tau\eta^2=10L\).  The fixed-reference-domain
 integral discretization permits the range \(0.1\geq\eta\geq10^{-20}\).}
 \label{fig:flower-eigenvalue-test}
\end{figure}

The common leading term obscures the band splitting at very small
\(\eta\). For the ordered numerical bands, write
\(\widehat\lambda_1=\widehat\lambda_-\),
\(\widehat\lambda_2=\widehat\lambda_+\), and set
\(\mu_1=\mu_+\), \(\mu_2=\mu_-\).
Subtracting the leading term and multiplying by \(L^2\) gives
\[
 L^2\left(\eta^2\widehat\lambda_j(\xi)
 -\frac{2\pi}{|D|L}\right)
 \longrightarrow-\frac{4\pi^2}{|D|^2}\mu_j(\xi).
\]
Figure~\ref{fig:flower-centered-test} displays this correction for the
same four momenta. The two centered branches approach the finite
limits predicted by the effective matrix and remain separated away
from \(K\); at \(K\), they coincide.

\begin{figure}[!htbp]
 \centering
 \includegraphics[width=\textwidth]{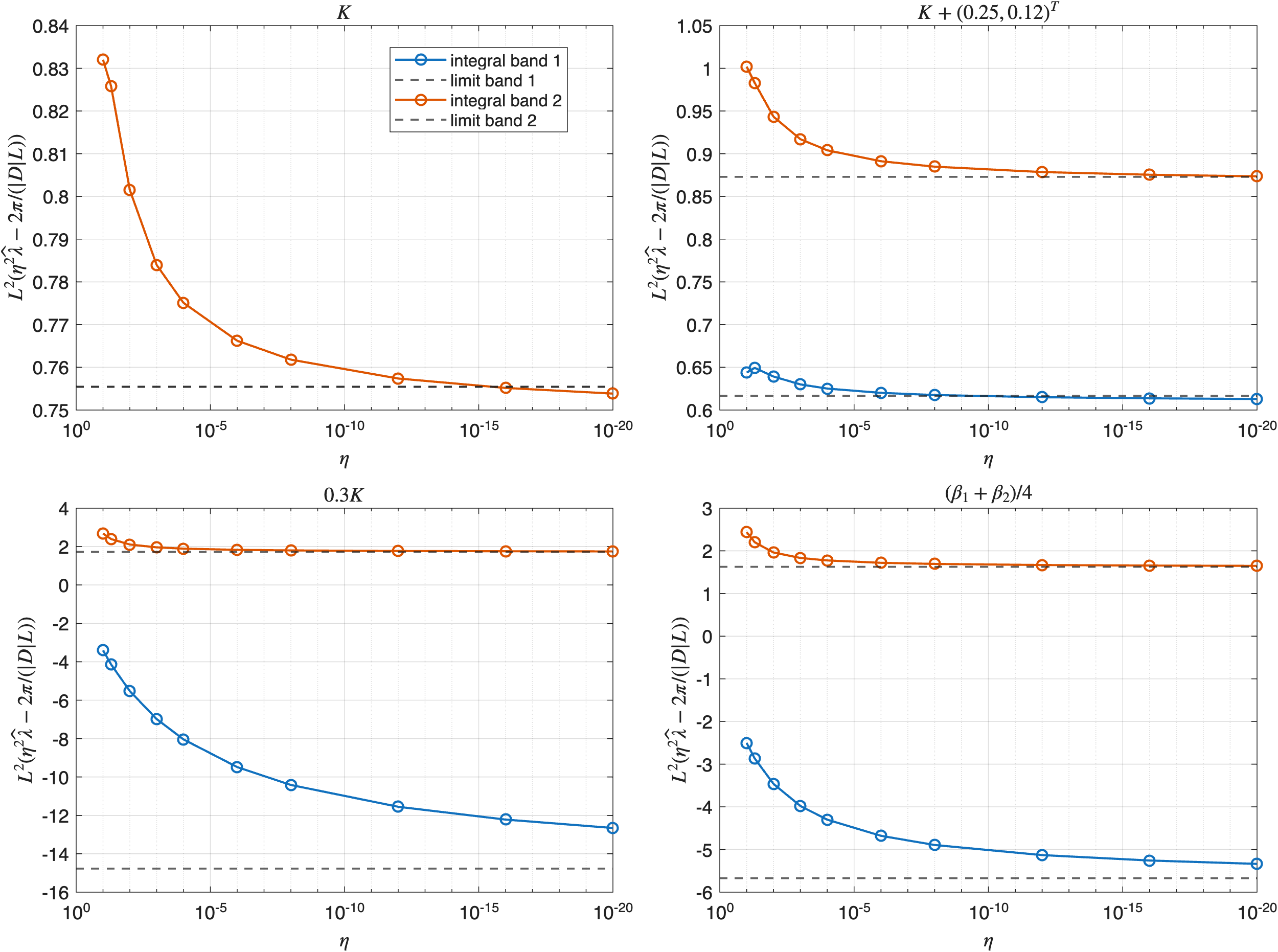}
 \caption{The centered second-order correction to the two scaled
 eigenvalues along \(\tau\eta^2=10L\).  Dashed lines are the limits
 \(-4\pi^2\mu_j(\xi)/|D|^2\).}
 \label{fig:flower-centered-test}
\end{figure}

\subsection{Dirac cone and conical slope}
\label{sec:numerical-dirac-cone}

For the slope comparisons below, fix the difference step
\[
 t_0=5\times10^{-4}
\]
and define
\[
 s_{\mathrm{num}}
 :=\frac{\widehat\lambda_+(K+t_0(1,0)^T)
 -\widehat\lambda_-(K+t_0(1,0)^T)}{2t_0},
 \qquad
 s_{\mathrm{asy}}
 :=\frac{4\pi^2|c|}{|D|\eta^2L^2},
 \qquad
 R_s:=\frac{s_{\mathrm{num}}}{s_{\mathrm{asy}}}.
\]
Here \(|c|\) is evaluated independently from the Ewald representation
in Appendix~\ref{app:c-nonzero}.  For the finite-difference quotient
of the exact eigenvalues, as \(t_0\to0\) with \(\epsilon,\eta\) fixed,
\eqref{equ:exact-Dirac-cone} gives
\[
 R_s=1+O\left(\frac1L+\frac1{\tau\eta^2L}\right)
 +\frac{O_{\epsilon,\eta}(t_0)}{s_{\mathrm{asy}}}.
\]
The last term is the finite-difference error; numerical evaluation
also introduces discretization error.
We next test the exact conical splitting in
\eqref{equ:exact-Dirac-cone}.  For \(\eta=10^{-4}\) and
\(\tau\eta^2=10L\), we compute the two lowest
eigenvalues along the line
\[
 \xi=K+t(1,0)^T,\qquad |t|\le0.02.
\]
The left panel of Figure~\ref{fig:numerical-dirac-cone} shows the
resulting crossing together with the leading asymptotic cone.  The
right panel shows the two numerical bands in the same neighborhood of
\(K\), in the coordinates \(\xi-K\).
Figure~\ref{fig:numerical-slope-error} plots \(R_s\) against \(1/L\)
for \(0.1\geq\eta\geq10^{-20}\), again along
\(\tau\eta^2=10L\).  On this path,
\[
 R_s=1+O\left(\frac1L+\frac1{10L^2}\right)
 +\frac{O_{\epsilon,\eta}(t_0)}{s_{\mathrm{asy}}},
\]
so the exact normalized slope tends to one as \(\eta\) decreases.
For the fixed-step quantity \(R_s\), the finite-difference error must
also be controlled before concluding convergence.
Numerically, the ratio decreases monotonically from \(1.408\) at
\(\eta=0.1\) to \(1.018\) at \(\eta=10^{-20}\).

\begin{figure}[!htbp]
 \centering
 \includegraphics[width=\textwidth]{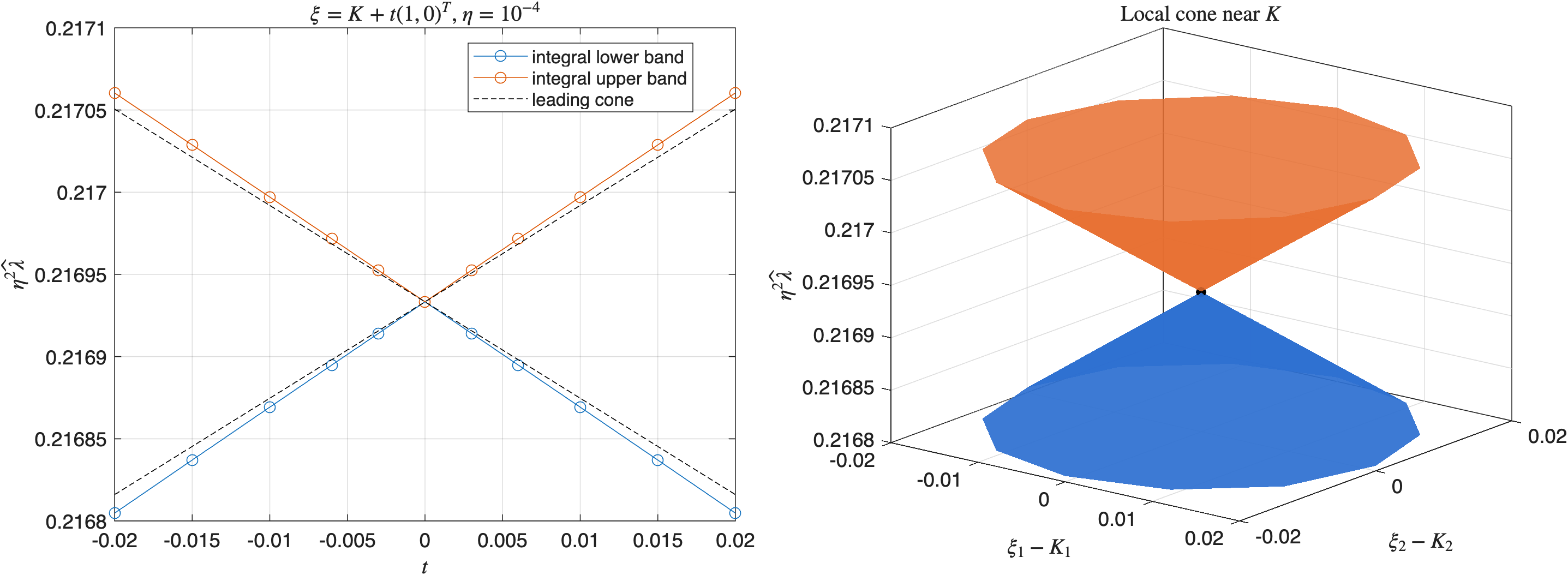}
 \caption{Numerical illustration of the Dirac cone.  Left: the two
 integral-operator bands along \(\xi=K+t(1,0)^T\) for \(\eta=10^{-4}\),
 compared with the leading cone in \eqref{equ:exact-Dirac-cone}.
 Right: the two numerical bands in the neighborhood
 \(\lvert\xi-K\rvert\leq0.02\), plotted in coordinates
 \(\xi-K\) to resolve the two conical sheets meeting at \(K\).}
 \label{fig:numerical-dirac-cone}
\end{figure}

\begin{figure}[!htbp]
 \centering
 \includegraphics[width=.68\textwidth]{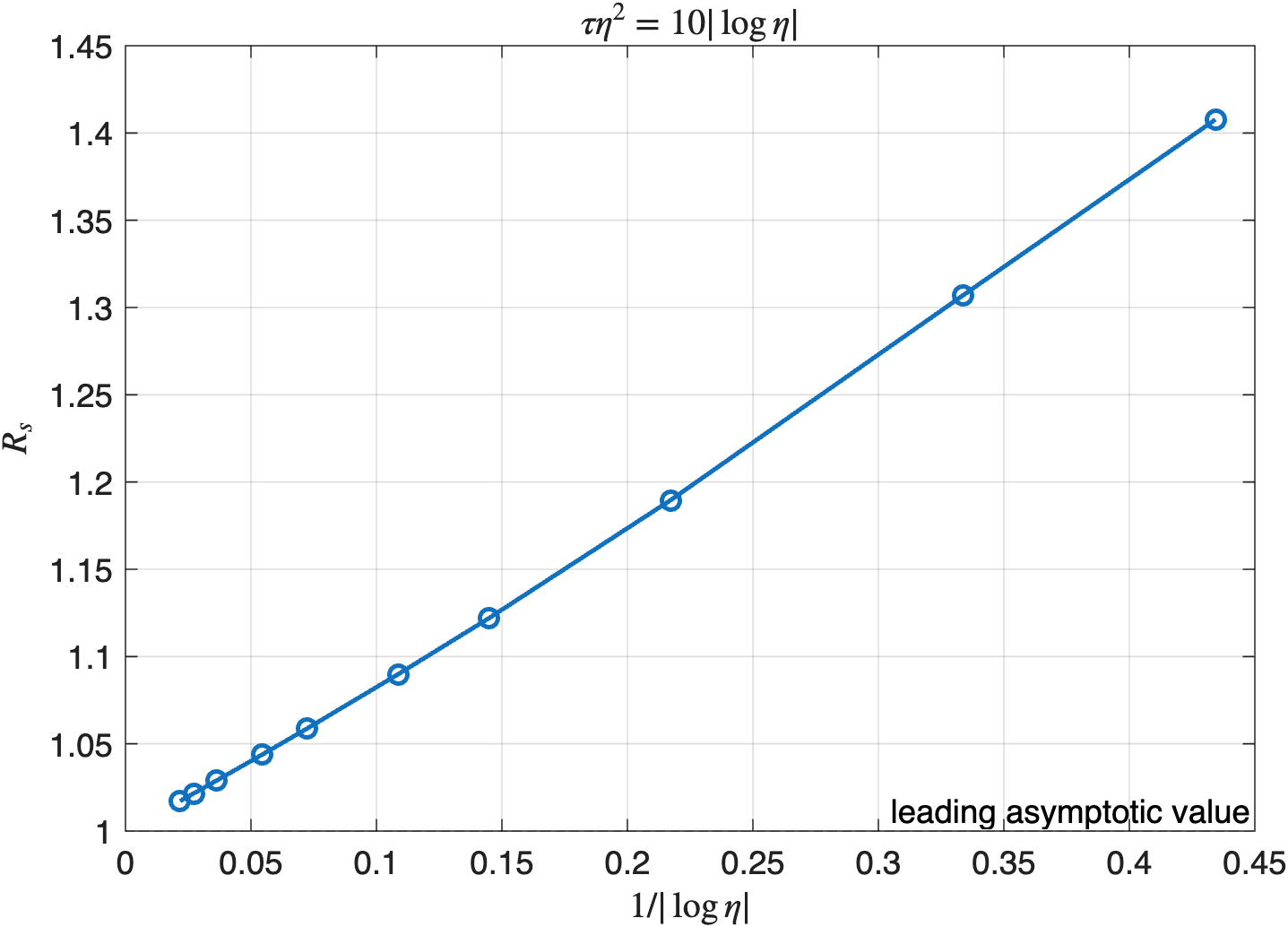}
 \caption{The normalized finite-difference slope \(R_s\) against
 \(1/|\log\eta|\) along \(\tau\eta^2=10|\log\eta|\), for
 \(0.1\geq\eta\geq10^{-20}\).  The exact normalized conical slope
 tends to one.}
 \label{fig:numerical-slope-error}
\end{figure}

Finally, to distinguish the two error contributions while keeping the
inclusion size fixed, we take \(\eta=10^{-8}\) and write
\[
 \alpha:=\tau\eta^2.
\]
The relative asymptotic error in the exact conical slope has the form
\(O(1/L+1/(\alpha L))\).  The first contribution is fixed in this
experiment, while the second decreases as \(\alpha\) grows.
Figure~\ref{fig:alpha-sweep} shows a trend consistent with these estimates: rapid
improvement for \(\alpha<1\), followed by a plateau for
\(\alpha\gtrsim5\).  Once this plateau is reached, increasing \(\tau\)
has little effect unless \(\eta\) is decreased as well.

\begin{figure}[!htbp]
 \centering
 \includegraphics[width=.68\textwidth]{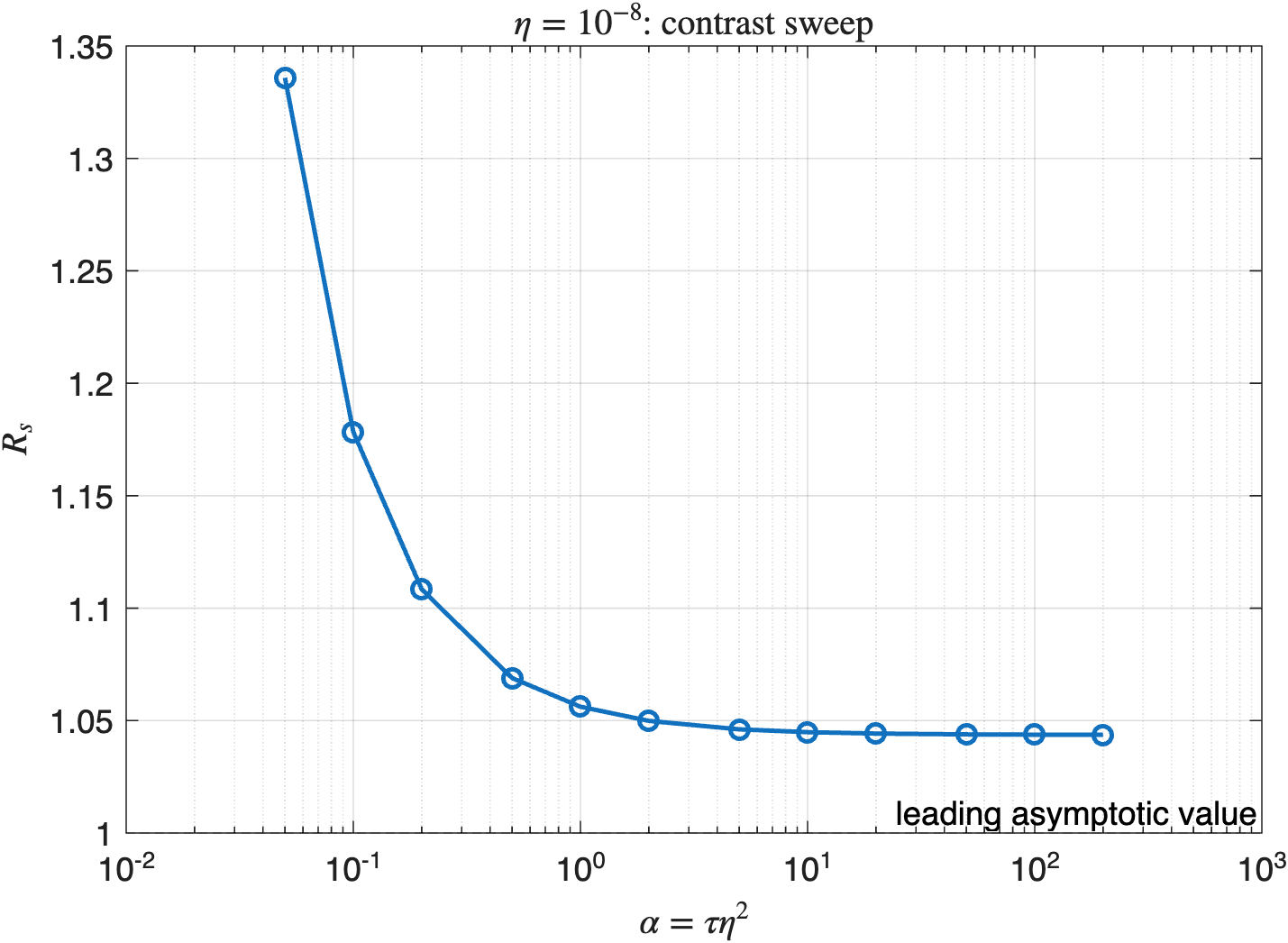}
 \caption{The normalized finite-difference slope \(R_s\) as
 \(\alpha=\tau\eta^2\) varies at the fixed size \(\eta=10^{-8}\).
 The plateau at large \(\alpha\) is consistent with the fixed
 inclusion-size contribution; finite-difference and discretization
 errors may also contribute.}
 \label{fig:alpha-sweep}
\end{figure}

\FloatBarrier
\section{Conclusion}
\label{sec:conclusion}

We have established the existence and continuation of Dirac points in
a two-dimensional honeycomb dielectric structure with two inclusions
per cell. In the joint small-inclusion and high-contrast regime, the
Lippmann--Schwinger formulation reduces the spectral problem near the
leading resonant scale to a two-dimensional interaction matrix. This
reduction yields a two-term eigenvalue expansion, uniform for
quasi-momenta separated from the reciprocal lattice, with explicit
dependence of the remainder on inclusion size and dielectric contrast.
At \(K\), rotation and inversion-conjugation symmetries enforce an
exact double eigenvalue. A Schur-complement analysis of the full
operator pencil, together with the nonvanishing Green-function
coefficient, proves a nonzero conical splitting and gives the leading
asymptotic slope in the scaled eigenvalue.

For each fixed sufficiently large contrast, this small-inclusion
Dirac point initiates a real-analytic eigenvalue branch throughout
the interval before inclusions touch. The continued branch remains
a nondegenerate Dirac point outside a relatively closed discrete
set of inclusion sizes, although its ordered band indices may change
at spectral crossings. Numerical computations for three-petal
inclusions are consistent with the eigenvalue and slope asymptotics,
using a fixed-domain Nystr\"om discretization and Fourier
cross-checks where feasible. Thus, the small-inclusion analysis provides
 a Dirac point from which analytic continuation establishes nondegenerate 
 Dirac crossings throughout the admissible size interval, except at a 
 discrete set of sizes.

\begin{appendices}
\section{\texorpdfstring{Nonvanishing of the constant \(c\)}{Nonvanishing of the constant c}}
\label{app:c-nonzero}

Set
\[
d:=z_1-z_2=\left(\frac1{\sqrt3},0\right)^T,
\qquad |Y|=\frac{\sqrt3}{2},
\qquad p_k:=k+K.
\]
Writing \(G_\lambda^\xi(z):=G_\lambda^\xi(z,0)\), we have
\begin{equation}\label{app:eq:c-Fourier}
\left.\nabla_\xi G_0^\xi\right|_{\xi=K}(d)
=\frac1{|Y|}\sum_{k\in\Lambda^*}e^{ip_k\cdot d}
\left(\frac{i\,d}{|p_k|^2}-\frac{2p_k}{|p_k|^4}\right).
\end{equation}
Since the first term is not absolutely summable, we use an Ewald
decomposition.  From
\[
\frac1{|p|^2}=\int_0^\infty e^{-t|p|^2}\,dt,
\qquad
\frac1{|p|^4}=\int_0^\infty t e^{-t|p|^2}\,dt,
\]
we split the integral at
\[
T:=\frac{\sqrt3}{8\pi}.
\]
The part over \([T,\infty)\) is
\begin{align}
V_{\mathrm{rec}}(T)
:={}&\frac1{|Y|}\sum_{k\in\Lambda^*}
e^{ip_k\cdot d-T|p_k|^2}
\notag\\[-1mm]
&\times\left[
\frac{i\,d}{|p_k|^2}
-2p_k\left(\frac{T}{|p_k|^2}+\frac1{|p_k|^4}\right)
\right].
\label{app:eq:reciprocal-part}
\end{align}
For the part over \((0,T]\), Poisson summation gives
\[
\frac1{|Y|}\sum_{k\in\Lambda^*}
e^{-t|k+\xi|^2}e^{i(k+\xi)\cdot d}
=\frac1{4\pi t}\sum_{\ell\in\Lambda}
e^{-i\xi\cdot\ell}e^{-|d+\ell|^2/(4t)}.
\]
Differentiating in \(\xi\), setting \(\xi=K\), and integrating in
\(t\), we obtain
\begin{equation}\label{app:eq:direct-part}
V_{\mathrm{dir}}(T)
:=-\frac{i}{4\pi}\sum_{\ell\in\Lambda}
\ell e^{-iK\cdot\ell}
E_1\left(\frac{|d+\ell|^2}{4T}\right),
\qquad
E_1(x):=\int_x^\infty\frac{e^{-s}}s\,ds.
\end{equation}
Thus
\begin{equation}\label{app:eq:Ewald}
\left.\nabla_\xi G_0^\xi\right|_{\xi=K}(d)
=V_{\mathrm{dir}}(T)+V_{\mathrm{rec}}(T),
\end{equation}
and both sums are absolutely and exponentially convergent.

We truncate both sums to \(|m|,|n|\le2\), and denote the resulting
first component by
\[
c_2:=\left[V_{\mathrm{dir},2}(T)+V_{\mathrm{rec},2}(T)\right]_1.
\]
To bound the tails, put \(j=\max\{|m|,|n|\}\).  Each square shell
contains \(8j\) points.  For
\(\ell=m\mathbf e_1+n\mathbf e_2\),
\[
|\ell|\le2j,
\qquad
|d+\ell|\ge
a_j:=\frac{\sqrt3}{2}j-\frac1{\sqrt3}.
\]
Using \(E_1(x)\le e^{-x}/x\), we get
\begin{equation}\label{app:eq:direct-tail-final}
\left|[V_{\mathrm{dir}}-V_{\mathrm{dir},2}]_1\right|
\le\frac{16T}{\pi}\sum_{j=3}^\infty
\frac{j^2}{a_j^2}e^{-a_j^2/(4T)}
<2.8541\times10^{-7}.
\end{equation}
Similarly, writing
\[
p_k=\left(m+\frac13\right)\boldsymbol\beta_1
+\left(n+\frac23\right)\boldsymbol\beta_2,
\]
we have
\[
|p_k|\ge
b_j:=\frac{4\pi}{\sqrt6}\left(j-\frac23\right).
\]
Hence
\begin{align}
\left|[V_{\mathrm{rec}}-V_{\mathrm{rec},2}]_1\right|
&\le\frac1{|Y|}\sum_{j=3}^\infty8j e^{-Tb_j^2}
\left(\frac{|d|}{b_j^2}+\frac{2T}{b_j}+\frac2{b_j^3}\right)
\notag\\
&<2.3821\times10^{-5}.
\label{app:eq:reciprocal-tail-final}
\end{align}
In both estimates, the displayed numerical bound follows by comparing
the positive series with a geometric series; the successive-term
ratios are bounded by \(2\times10^{-7}\) and \(3\times10^{-5}\),
respectively.

It remains to evaluate two finite sums of \(25\) terms.  Their phases
are third roots of unity since
\[
e^{-iK\cdot(m\mathbf e_1+n\mathbf e_2)}
=e^{-2\pi i(m+2n)/3},
\qquad
e^{i(m\boldsymbol\beta_1+n\boldsymbol\beta_2+K)\cdot d}
=e^{2\pi i(m+n+1)/3}.
\]
A directed-rounding interval evaluation gives
\[
\begin{aligned}
\bigl[V_{\mathrm{dir},2}\bigr]_1
&\in[-10^{-15},10^{-15}]
-i[0.010565785189922,0.010565785189925],\\
\bigl[V_{\mathrm{rec},2}\bigr]_1
&\in-[7.68616,7.68620]\times10^{-10}
-i[0.030800294748640,0.030800294748644].
\end{aligned}
\]
The calculation uses
outward-rounded decimal interval arithmetic, Machin's formula for
\(\pi\), Euler--Maclaurin summation with an integral-remainder bound
for Euler's constant, and the convergent power series for \(E_1\)
with an alternating-tail bound.
Consequently,
\[
|c_2|>0.04136607993856,
\qquad
|c-c_2|<2.411\times10^{-5}.
\]
The reverse triangle inequality now yields
\[
|c|\ge |c_2|-|c-c_2|>0.04134>0.
\]
Therefore \(c\ne0\).

\section{\texorpdfstring{The exact conical slope at \(K\)}{The exact conical slope at K}}
\label{app:exact-Dirac-cone}
We prove \eqref{equ:exact-Dirac-cone} under the hypotheses of
Theorem~\ref{thm:refined-Dirac-point}.  Throughout this appendix,
\(\widehat\lambda_0=\widehat\lambda_0^\epsilon(\eta)\),
\(L=|\log\eta|\), and inner products are those of
\(L^2(D_\eta)\), conjugate-linear in the first argument.
Choose normalized kernel vectors \(\phi_1\in L_1\) and
\(\phi_2=P_\eta\phi_1\in L_2\).
Since \(\Pi\) commutes with the rescaled rotation and
\(\mathcal E_0\cap U_\eta L_j=\mathbb C e_j\), the constant-mode
estimate, with the phase fixed by
\(\langle e_1,U_\eta\phi_1\rangle>0\), gives
\[
 \left\|\phi_j-
 \frac{\chi_{D_{\eta,j}}}{\eta\sqrt{|D|}}\right\|_{L^2(D_\eta)}
 =O(\delta_{\epsilon,\eta}),\qquad j=1,2.
\]
Also
\(\widehat\lambda_0\eta^2=O(L^{-1})\) and
\(\epsilon\widehat\lambda_0=o(1)\).

We first establish some estimates that will be used in the proof below. 
For \(\xi\) in a fixed neighborhood of \(K\) and \(\omega\) near
zero, the free spectral gap ensures analyticity in \(\omega\).
The Fourier series for \(\partial_\omega G_\omega^\xi\) and
\(\partial_\omega\nabla_\xi G_\omega^\xi\) converge absolutely
and uniformly: their terms are bounded by
\(C|k+\xi|^{-4}\) and
\(C(|k+\xi|^{-4}+|k+\xi|^{-5})\), respectively, for bounded
\(x,y\).  Hence
\[
 \nabla_\xi(G_\omega^\xi-G_0^\xi)
 =\int_0^\omega\partial_t\nabla_\xi G_t^\xi\,dt
 =O(|\omega|),\qquad
 \partial_\omega G_\omega^\xi=O(1).
\]
To bound \(\nabla_\xi G_0^\xi\), differentiate its Fourier
representation in the distributional sense, obtaining, away
from the spatial singularities,
\[
 \nabla_\xi G_0^\xi(x,y)
 =i(x-y)G_0^\xi(x,y)
 -\frac2{|Y|}\sum_{k\in\Lambda^*}
 \frac{(k+\xi)e^{i(k+\xi)\cdot(x-y)}}{|k+\xi|^4}.
\]
The series is absolutely uniformly convergent, since
\(\sum_k|k+\xi|^{-3}\le C\) in two dimensions.
For sufficiently small \(\eta\), differences of points in
\(D_\eta\) stay uniformly away from nonzero lattice points.
The local logarithmic expansion and boundedness away from the
diagonal therefore give
\[
 |(x-y)G_0^\xi(x,y)|
 \le C|x-y|\bigl(1+|\log|x-y||\bigr)\le C,
 \qquad x,y\in D_\eta,\quad x\ne y.
\]
Consequently, uniformly in the parameters under consideration,
\[
 |\nabla_\xi G_\omega^\xi(x,y)|
 \le |\nabla_\xi G_0^\xi(x,y)|+C|\omega|\le C.
\]
The differentiated kernel represents
\(\nabla_\xi\mathcal T_{D_\eta}^{\xi,\omega}\).  Its
Hilbert--Schmidt bound gives
\[
 \|\nabla_\xi\mathcal T_{D_\eta}^{\xi,\omega}\|
 \le\left(\int_{D_\eta}\int_{D_\eta}
 |\nabla_\xi G_\omega^\xi(x,y)|^2\,dx\,dy\right)^{1/2}
 \le C|D_\eta|=2C|D|\eta^2.
\]
Applying the same estimate to the bounded kernel
\(\partial_\omega G_\omega^\xi\) proves
\begin{equation}
 \|\nabla_\xi\mathcal T_{D_\eta}^{\xi,\omega}\|
 +\|\partial_\omega\mathcal T_{D_\eta}^{\xi,\omega}\|
 \le C\eta^2.
 \label{app:kernel-derivative-bounds}
\end{equation}

Now recall the operator pencil from
Theorem~\ref{thm:two_characteristic_values},
\[
 \widehat{\mathcal A}(\widehat\lambda,\xi;\epsilon,\eta)
 =I-\widehat\lambda\mathcal T_{D_\eta}^{\xi,\epsilon\widehat\lambda}.
\]
For fixed \(\epsilon,\eta\), it is real analytic in
\((\widehat\lambda,\xi)\) in operator norm. Write
\(\Delta\widehat\lambda=\widehat\lambda-\widehat\lambda_0\) and
\(h=\xi-K\). Expanding at the characteristic pair
\((\widehat\lambda_0,K)\) gives
\[
 \widehat{\mathcal A}(\widehat\lambda_0+\Delta\widehat\lambda,
 K+h)
 =\widehat{\mathcal A}(\widehat\lambda_0,K)
 +\Delta\widehat{\mathcal A}
 +O_{\epsilon,\eta}(|\Delta\widehat\lambda|^2+|h|^2),
\]
where \(\epsilon,\eta\) are suppressed and
\[
\begin{aligned}
 \Delta\widehat{\mathcal A}
 ={}&-\left(\mathcal T_{D_\eta}^{K,\epsilon\widehat\lambda_0}
 +\epsilon\widehat\lambda_0
 \left.\partial_\omega\mathcal T_{D_\eta}^{K,\omega}
 \right|_{\omega=\epsilon\widehat\lambda_0}\right)\Delta\widehat\lambda-\widehat\lambda_0 h\cdot
 \left.\nabla_\xi\mathcal T_{D_\eta}^{\xi,\epsilon\widehat\lambda_0}
 \right|_{\xi=K}.
\end{aligned}
\]
Let
\[
 \Pi_\parallel\psi=\sum_{j=1}^2(\phi_j,\psi)\phi_j,
 \qquad \Pi_\perp=I-\Pi_\parallel.
\]
Since the base pencil is self-adjoint and has kernel
\(\operatorname{span}\{\phi_1,\phi_2\}\),
\[
 \widehat{\mathcal A}(\widehat\lambda_0,K)\Pi_\parallel
 =\Pi_\parallel\widehat{\mathcal A}(\widehat\lambda_0,K)=0.
\]
Its restriction to \(\operatorname{Ran}\Pi_\perp\) is invertible. 
For \(\psi_\parallel=\Pi_\parallel\psi\) and
\(\psi_\perp=\Pi_\perp\psi\), the exact equation is
\[
 \begin{cases}
 \Pi_\parallel\widehat{\mathcal A}\Pi_\parallel\psi_\parallel
 +\Pi_\parallel\widehat{\mathcal A}\Pi_\perp\psi_\perp=0,\\
 \Pi_\perp\widehat{\mathcal A}\Pi_\parallel\psi_\parallel
 +\Pi_\perp\widehat{\mathcal A}\Pi_\perp\psi_\perp=0.
 \end{cases}
\]
Solving the second equation gives
\[
 \psi_\perp
 =-(\Pi_\perp\widehat{\mathcal A}\Pi_\perp)^{-1}
 \Pi_\perp\widehat{\mathcal A}\Pi_\parallel\psi_\parallel.
\]
Substitution into the first gives the equivalent Schur equation
\begin{equation}\label{app:eq:exact-schur}
 \left[\Pi_\parallel\widehat{\mathcal A}\Pi_\parallel
 -\Pi_\parallel\widehat{\mathcal A}\Pi_\perp
 (\Pi_\perp\widehat{\mathcal A}\Pi_\perp)^{-1}
 \Pi_\perp\widehat{\mathcal A}\Pi_\parallel\right]\psi_\parallel=0.
\end{equation}
Here \(\psi\ne0\) if and only if \(\psi_\parallel\ne0\).
Both cross blocks are
\(O_{\epsilon,\eta}(|\Delta\widehat\lambda|+|h|)\), whereas
the inverse block is bounded.  Thus
\[
 \left[\Pi_\parallel\Delta\widehat{\mathcal A}\Pi_\parallel
 +O_{\epsilon,\eta}(|\Delta\widehat\lambda|^2+|h|^2)\right]
 \psi_\parallel=0,\qquad \psi_\parallel\ne0.
\]
In the orthonormal basis \(\{\phi_1,\phi_2\}\), this is equivalent to
\[
 \det\begin{pmatrix}
 (\phi_1,\Delta\widehat{\mathcal A}\phi_1)&
 (\phi_1,\Delta\widehat{\mathcal A}\phi_2)\\
 (\phi_2,\Delta\widehat{\mathcal A}\phi_1)&
 (\phi_2,\Delta\widehat{\mathcal A}\phi_2)
 \end{pmatrix}
 +O_{\epsilon,\eta}(|\Delta\widehat\lambda|^3+|h|^3)=0.
\]

Rotation symmetry makes the spectral-parameter derivative
diagonal in \(\{\phi_1,\phi_2\}\), and inversion conjugation makes
its diagonal entries equal.
For the momentum derivative, write
\(T_h:=\mathcal T_{D_\eta}^{K+h,\epsilon\widehat\lambda_0}\) and
\[
 \Theta_h:=e^{-2ih\cdot\mathbf e_1}\chi_{D_{\eta,1}}
          +e^{-ih\cdot\mathbf e_1}\chi_{D_{\eta,2}}.
\]
A change of variables under the physical rotation gives
\[
 T_{Rh}=\Theta_h R_\eta T_h R_\eta^{-1}\Theta_h^{-1}.
\]
Differentiating at \(h=0\) introduces, in addition to the
rotated derivative, the commutator \([B(h),T_0]\), where
\(B(h):=-i(h\cdot\mathbf e_1)
(2\chi_{D_{\eta,1}}+\chi_{D_{\eta,2}})\).
Its compression to the kernel vanishes because
\(T_0\Pi_\parallel=\Pi_\parallel T_0
=\widehat\lambda_0^{-1}\Pi_\parallel\).
Thus the compressed derivative transforms by conjugation with
\(\operatorname{diag}(\kappa,\kappa^2)\): its diagonal entries
are rotation-invariant linear functions of \(h\), hence zero,
and its upper-right entry transforms by \(\kappa^{-1}\), hence
is a multiple of \(h_1+ih_2\).  Self-adjointness makes the
lower-left entry its conjugate.
The determinant equation consequently becomes
\begin{equation}\label{app:eq:cone-determinant}
\begin{aligned}
 &(\Delta\widehat\lambda)^2
 \left(\phi_1,\left[
 \mathcal T_{D_\eta}^{K,\epsilon\widehat\lambda_0}
 +\epsilon\widehat\lambda_0
 \left.\partial_\omega\mathcal T_{D_\eta}^{K,\omega}
 \right|_{\omega=\epsilon\widehat\lambda_0}
 \right]\phi_1\right)^2\\
 &\quad-\widehat\lambda_0^2|h|^2
 \left|\left(\phi_1,
 \left.\partial_{\xi_1}\mathcal T_{D_\eta}^{\xi,\epsilon\widehat\lambda_0}
 \right|_{\xi=K}\phi_2\right)\right|^2
 +O_{\epsilon,\eta}(|\Delta\widehat\lambda|^3+|h|^3)=0.
\end{aligned}
\end{equation}

The kernel equation and \eqref{app:kernel-derivative-bounds} give,
for \(j=1,2\),
\[
\begin{aligned}
 &\left(\phi_j,\left[
 \mathcal T_{D_\eta}^{K,\epsilon\widehat\lambda_0}
 +\epsilon\widehat\lambda_0
 \left.\partial_\omega\mathcal T_{D_\eta}^{K,\omega}
 \right|_{\omega=\epsilon\widehat\lambda_0}
 \right]\phi_j\right)\\
 &\qquad=\frac1{\widehat\lambda_0}+O(\epsilon/L)
 =\frac{|D|\eta^2L}{2\pi}
 \left[1+O\left(\frac1L+\frac{\epsilon}{\eta^2L^2}\right)\right]>0.
\end{aligned}
\]
For the off-diagonal entry, the constant-mode approximation,
\eqref{app:kernel-derivative-bounds}, and Taylor expansion
of the spatially smooth kernel near \((z_1,z_2)\) give
\[
\begin{aligned}
 &\left(\phi_1,
 \left.\nabla_\xi\mathcal T_{D_\eta}^{\xi,\epsilon\widehat\lambda_0}
 \right|_{\xi=K}\phi_2\right)\\
 &\quad=\frac1{\eta^2|D|}
 \int_{D_{\eta,1}}\int_{D_{\eta,2}}
 \left.\nabla_\xi G_{\epsilon\widehat\lambda_0}^\xi
 \right|_{\xi=K}(x,y)\,dy\,dx
 +O(\eta^2\delta_{\epsilon,\eta})\\
 &\quad=\frac{\eta^2}{|D|}\int_D\int_D
 \left.\nabla_\xi G_{\epsilon\widehat\lambda_0}^\xi
 \right|_{\xi=K}(z_1+\eta x,z_2-\eta y)\,dx\,dy
 +O(\eta^2\delta_{\epsilon,\eta})\\
 &\quad=\eta^2|D|c\begin{pmatrix}1\\i\end{pmatrix}
 +O(\eta^3+\epsilon/L+\eta^2\delta_{\epsilon,\eta})\\
 &\quad=\eta^2|D|c\begin{pmatrix}1\\i\end{pmatrix}
 +O(\eta^2/L+\epsilon/L).
\end{aligned}
\]

Define
\begin{equation}\label{app:eq:exact-slope}
 v_{\epsilon,\eta}:=
 \frac{\widehat\lambda_0
 \left|\left(\phi_1,
 \left.\partial_{\xi_1}\mathcal T_{D_\eta}^{\xi,\epsilon\widehat\lambda_0}
 \right|_{\xi=K}\phi_2\right)\right|}
 {\left(\phi_1,\left[
 \mathcal T_{D_\eta}^{K,\epsilon\widehat\lambda_0}
 +\epsilon\widehat\lambda_0
 \left.\partial_\omega\mathcal T_{D_\eta}^{K,\omega}
 \right|_{\omega=\epsilon\widehat\lambda_0}
 \right]\phi_1\right)}
 =\frac{4\pi^2|c|}{|D|\eta^2L^2}
 +O\left(\frac1{\eta^2L^3}+\frac{\epsilon}{\eta^4L^3}\right)>0.
\end{equation}
Let \(S(\delta,h)\) be the matrix of the operator in brackets in
\eqref{app:eq:exact-schur}, in the basis \(\{\phi_1,\phi_2\}\),
where \(\delta=\Delta\widehat\lambda\), and let \(s>0\) be the
denominator in \eqref{app:eq:exact-slope}. Set
\(h=r\omega\), \(|\omega|=1\), and \(\delta=r\mu\).
Since \(S\) is analytic and \(S(0,0)=0\), the function
\(F(r,\mu,\omega):=r^{-2}\det S(r\mu,r\omega)\) extends analytically
to \(r=0\). By \eqref{app:eq:cone-determinant} and
\eqref{app:eq:exact-slope},
\[
 F(0,\mu,\omega)=s^2(\mu^2-v_{\epsilon,\eta}^2),
 \qquad
 \partial_\mu F(0,\pm v_{\epsilon,\eta},\omega)
 =\pm2s^2v_{\epsilon,\eta}\ne0.
\]
The implicit function theorem therefore gives two real roots
\(\mu_\pm(r,\omega)=\pm v_{\epsilon,\eta}+O_{\epsilon,\eta}(r)\),
uniformly in \(\omega\) by compactness of the unit circle.
By the Schur-complement equivalence and
Theorem~\ref{thm:two_characteristic_values}, these yield the two
characteristic values already counted. Multiplying by \(r=|h|\) gives
\[
 \widehat\lambda_\pm(K+h)=\widehat\lambda_0
 \pm v_{\epsilon,\eta}|h|+O_{\epsilon,\eta}(|h|^2),
\]
which, together with \eqref{app:eq:exact-slope}, proves
\eqref{equ:exact-Dirac-cone}.

\end{appendices}
\newpage
\bibliographystyle{plain}
\bibliography{highcon}

@book{reed2012methods,
  title={Methods of modern mathematical physics: Functional analysis},
  author={Reed, Michael},
  year={2012},
  publisher={Elsevier}
}

@book{lewin2024spectral,
  title={Spectral theory and quantum mechanics},
  author={Lewin, Mathieu},
  year={2024},
  publisher={Springer}
}

@book{kato1995perturbation,
  title={Perturbation Theory for Linear Operators},
  author={Kato, Tosio},
  edition={2},
  year={1995},
  publisher={Springer},
  address={Berlin}
}

@article{introWallace1947,
 author={Wallace, P. R.}, title={The Band Theory of Graphite},
 journal={Physical Review}, year={1947}, volume={71}, pages={622--634},
 doi={10.1103/PhysRev.71.622}, url={https://doi.org/10.1103/PhysRev.71.622}
}

@article{introNovoselov2005,
 author={Novoselov, K. S. and Geim, A. K. and Morozov, S. V. and Jiang, D. and Katsnelson, M. I. and Grigorieva, I. V. and Dubonos, S. V. and Firsov, A. A.},
 title={Two-dimensional gas of massless {Dirac} fermions in graphene}, journal={Nature}, year={2005}, volume={438}, pages={197--200},
 doi={10.1038/nature04233}, url={https://doi.org/10.1038/nature04233}
}

@article{introCastroNeto2009,
 author={Castro Neto, A. H. and Guinea, F. and Peres, N. M. R. and Novoselov, K. S. and Geim, A. K.},
 title={The electronic properties of graphene}, journal={Reviews of Modern Physics}, year={2009}, volume={81}, pages={109--162},
 doi={10.1103/RevModPhys.81.109}, url={https://arxiv.org/abs/0709.1163}
}

@article{introPeleg2007,
 author={Peleg, Or and Bartal, Guy and Freedman, Barak and Manela, Ofer and Segev, Mordechai and Christodoulides, Demetrios N.},
 title={Conical Diffraction and Gap Solitons in Honeycomb Photonic Lattices}, journal={Physical Review Letters}, year={2007}, volume={98}, pages={103901},
 doi={10.1103/PhysRevLett.98.103901}, url={https://doi.org/10.1103/PhysRevLett.98.103901}
}

@article{introOchiaiOnoda2009,
 author={Ochiai, Tetsuyuki and Onoda, Masaru}, title={Photonic analog of graphene model and its extension: {Dirac} cone, symmetry, and edge states}, journal={Physical Review B}, year={2009}, volume={80}, pages={155103},
 doi={10.1103/PhysRevB.80.155103}, url={https://arxiv.org/abs/0905.4149}
}

@article{introHaldaneRaghu2008,
 author={Haldane, F. D. M. and Raghu, S.}, title={Possible Realization of Directional Optical Waveguides in Photonic Crystals with Broken Time-Reversal Symmetry}, journal={Physical Review Letters}, year={2008}, volume={100}, pages={013904},
 doi={10.1103/PhysRevLett.100.013904}, url={https://doi.org/10.1103/PhysRevLett.100.013904}
}

@article{introRechtsman2013,
 author={Rechtsman, Mikael C. and Zeuner, Julia M. and Plotnik, Yonatan and Lumer, Yaakov and Podolsky, Daniel and Dreisow, Felix and Nolte, Stefan and Segev, Mordechai and Szameit, Alexander}, title={Photonic {Floquet} topological insulators}, journal={Nature}, year={2013}, volume={496}, pages={196--200},
 doi={10.1038/nature12066}, url={https://doi.org/10.1038/nature12066}
}

@article{introKuchment2016,
 author={Kuchment, Peter}, title={An overview of periodic elliptic operators}, journal={Bulletin of the American Mathematical Society}, year={2016}, volume={53}, pages={343--414},
 doi={10.1090/bull/1528}, url={https://arxiv.org/abs/1510.00971}
}

@article{introFeffermanWeinstein2012,
 author={Fefferman, Charles L. and Weinstein, Michael I.}, title={Honeycomb lattice potentials and {Dirac} points}, journal={Journal of the American Mathematical Society}, year={2012}, volume={25}, pages={1169--1220},
 doi={10.1090/S0894-0347-2012-00745-0}, url={https://arxiv.org/abs/1202.3839}
}

@article{introBerkolaikoComech2018,
 author={Berkolaiko, Gregory and Comech, Andrew}, title={Symmetry and {Dirac} points in graphene spectrum}, journal={Journal of Spectral Theory}, year={2018}, volume={8}, pages={1099--1147}, url={https://ems.press/journals/jst/articles/15625}
}

@article{introFeffermanWeinstein2014,
 author={Fefferman, Charles L. and Weinstein, Michael I.}, title={Wave Packets in Honeycomb Structures and Two-Dimensional {Dirac} Equations}, journal={Communications in Mathematical Physics}, year={2014}, volume={326}, pages={251--286},
 doi={10.1007/s00220-013-1847-2}, url={https://arxiv.org/abs/1212.6072}
}

@article{introKuchmentPost2007,
 author={Kuchment, Peter and Post, Olaf}, title={On the Spectra of Carbon Nano-Structures}, journal={Communications in Mathematical Physics}, year={2007}, doi={10.1007/s00220-007-0316-1}, url={https://arxiv.org/abs/math-ph/0612021}
}

@article{introLee2016,
 author={Lee, Minjae}, title={{Dirac} Cones for Point Scatterers on a Honeycomb Lattice}, journal={SIAM Journal on Mathematical Analysis}, year={2016}, pages={1459--1488},
 doi={10.1137/14095827X}, url={https://doi.org/10.1137/14095827X}
}

@article{introFeffermanLeeThorpWeinstein2018,
 author={Fefferman, Charles L. and Lee-Thorp, James P. and Weinstein, Michael I.}, title={Honeycomb {Schr\"odinger} Operators in the Strong Binding Regime}, journal={Communications on Pure and Applied Mathematics}, year={2018}, volume={71}, pages={1178--1270},
 doi={10.1002/cpa.21735}, url={https://arxiv.org/abs/1610.04930}
}

@article{introLeeThorpWeinsteinZhu2019,
 author={Lee-Thorp, J. P. and Weinstein, M. I. and Zhu, Y.}, title={Elliptic Operators with Honeycomb Symmetry: {Dirac} Points, Edge States and Applications to Photonic Graphene}, journal={Archive for Rational Mechanics and Analysis}, year={2019}, volume={232}, pages={1--63},
 doi={10.1007/s00205-018-1315-4}, url={https://arxiv.org/abs/1710.03389}
}

@article{introFeffermanLeeThorpWeinstein2016,
 author={Fefferman, C. L. and Lee-Thorp, J. P. and Weinstein, M. I.}, title={Edge States in Honeycomb Structures}, journal={Annals of PDE}, year={2016}, volume={2}, pages={12},
 doi={10.1007/s40818-016-0015-3}, url={https://doi.org/10.1007/s40818-016-0015-3}
}

@article{introCassierWeinstein2021,
 author={Cassier, Maxence and Weinstein, Michael I.}, title={High Contrast Elliptic Operators in Honeycomb Structures}, journal={Multiscale Modeling \& Simulation}, year={2021}, volume={19}, pages={1784--1856},
 doi={10.1137/21M1408968}, url={https://arxiv.org/abs/2103.16682}
}

@article{introAmmariBubbles2020,
 author={Ammari, Habib and Fitzpatrick, Brian and Hiltunen, Erik Orvehed and Lee, Hyundae and Yu, Sanghyeon}, title={Honeycomb-Lattice {Minnaert} Bubbles}, journal={SIAM Journal on Mathematical Analysis}, year={2020}, volume={52}, pages={5441--5466},
 doi={10.1137/19M1281782}, url={https://arxiv.org/abs/1811.03905}
}

@article{introAmmariHomogenization2020,
 author={Ammari, Habib and Hiltunen, Erik Orvehed and Yu, Sanghyeon}, title={A High-Frequency Homogenization Approach Near the {Dirac} Points in Bubbly Honeycomb Crystals}, journal={Archive for Rational Mechanics and Analysis}, year={2020}, volume={238}, pages={1559--1583},
 doi={10.1007/s00205-020-01572-w}, url={https://arxiv.org/abs/1812.06178}
}

@article{introLiLinZhang2023,
 author={Li, Wei and Lin, Junshan and Zhang, Hai}, title={{Dirac} Points for the Honeycomb Lattice with Impenetrable Obstacles}, journal={SIAM Journal on Applied Mathematics}, year={2023}, volume={83}, pages={1546--1571},
 doi={10.1137/22M1505116}, url={https://doi.org/10.1137/22M1505116}
}

@article{introAmmariDielectric2019,
 author={Ammari, Habib and Dabrowski, Alexander and Fitzpatrick, Brian and Millien, Pierre and Sini, Mourad}, title={Subwavelength resonant dielectric nanoparticles with high refractive indices}, journal={Mathematical Methods in the Applied Sciences}, year={2019}, volume={42}, pages={6567--6579},
 doi={10.1002/mma.5760}, url={https://doi.org/10.1002/mma.5760}
}

@article{introAmmariLiZou2023,
 author={Ammari, Habib and Li, Bowen and Zou, Jun}, title={Mathematical analysis of electromagnetic scattering by dielectric nanoparticles with high refractive indices}, journal={Transactions of the American Mathematical Society}, year={2023}, volume={376}, pages={39--90},
 doi={10.1090/tran/8641}, url={https://arxiv.org/abs/2003.10223}
}

@book{ammari2018mathematical,
  title={Mathematical and computational methods in photonics and phononics},
  author={Ammari, Habib and Fitzpatrick, Brian and Kang, Hyeonbae and Ruiz, Matias and Yu, Sanghyeon and Zhang, Hai},
  volume={235},
  year={2018},
  publisher={American Mathematical Soc.}
}
\end{document}